\documentclass{amsart}

\usepackage{latexsym,amsfonts}
\usepackage{amssymb}
\usepackage{verbatim}
\usepackage[english]{babel}
\usepackage{mathrsfs} 

\DeclareMathOperator{\XGCD}{XGCD}

\def\tr{\mathrm{tr}}
\def\C{\mathscr{C}}
\def\S{\mathscr{S}}
\def\B{\mathscr{B}}
\def\Cl{\mathcal{C}}

\def\Nb{\mathcal{N}}
\def\Jor{\mathrm{Jor}}

\def\diag{\mathrm{diag}}
\def\GL{\mathrm{GL}}    
\def\Mat{\mathrm{Mat}}   
\def\I{\mathrm{I}}
\def\id{\mathrm{id}}
\def\Ker{\mathrm{Ker}}
\def\im{\mathrm{Im}}

\def\PSU{\mathrm{PSU}}
\def\PSL{\mathrm{PSL}}
\def\PSp{\mathrm{PSp}}

\def\SL{\mathrm{SL}}  
\def\SU{\mathrm{SU}}
\def\Sp{\mathrm{Sp}}
\def\Sym{\mathrm{Sym}}
\def\U{\mathrm{U}}

\def\P{{\rm P}}
\def\equad{\quad \textrm{ and } \quad}

\newtheorem{theorem}{Theorem}[section] 
\newtheorem{lemma}[theorem]{Lemma}     

\newtheorem{proposition}[theorem]{Proposition}

\theoremstyle{definition}

\newcommand{\Z}{\mathbb{Z}}   
\newcommand{\Q}{\mathbb{Q}}    
\newcommand{\F}{\mathbb{F}}    
\newcommand{\zeri}[1]{\mathbf{0}_{#1}}

\numberwithin{equation}{section}

\begin{document}

\title{The $(2,3)$-generation of the finite unitary groups}

\author{M.A. Pellegrini}
\email{marcoantonio.pellegrini@unicatt.it}

\author{M.C. Tamburini Bellani}
\email{mariaclara.tamburini@gmail.com}

\address{Dipartimento di Matematica e Fisica, Universit\`a Cattolica del Sacro Cuore,
Via Musei 41, 25121 Brescia, Italy}

\begin{abstract}
In this paper we prove that the unitary groups $\SU_n(q^2)$ are $(2,3)$-generated
for any prime power $q$
and any integer $n\geq 8$.
By previous results this implies that, if $n\geq 3$, the groups $\SU_n(q^2)$ and  $\PSU_n(q^2)$ are 
$(2,3)$-generated,
except when $(n,q)\in\{(3,2),(3,3),(3,5),(4,2), (4,3),(5,2)\}$.
\end{abstract}


\maketitle

\section{Introduction}

A group  is said to be $(2,3)$-generated if it can be generated by an element of order $2$ and an element of order
$3$. It is well known that such groups are epimorphic  images of the infinite unimodular group $\PSL_2(\Z)$.
By a famous result of Liebeck and Shalev \cite{LS}, the finite classical simple groups are
$(2,3)$-generated, apart from the two infinite families $\PSp_4(q)$ with $q=2^f,3^f$ and a finite list 
$\mathcal{L}$ of exceptions.
This list $\mathcal{L}$ includes the ten groups $\PSL_2(9), \PSL_3(4), \PSU_3(9), \PSU_3(25), \PSL_4(2), \PSU_4(4), 
\PSU_4(9),$ $\PSU_5(4)$ (see \cite{Ischia} and the references therein), $\P\Omega_8^+(2), \P\Omega_8^ +(3)$ (see 
\cite{Max}). However, the problems of determining whether other 
exceptions exist and  finding $(2,3)$-generating pairs for the positive cases are still open (see \cite[Problem 
18.98]{KM}).

It is known that $\mathcal{L}$ contains no other linear group $\PSL_n(q)$ and no other classical simple group of 
dimension less than $8$ (see \cite{SL12,Ischia}). 
Here, in particular, we show that $\mathcal{L}$ contains no other unitary group 
$\PSU_n(q^2)$.
More precisely, we give a constructive proof of  the following result.

\begin{theorem}\label{main}
The groups $\SU_n(q^2)$ are $(2,3)$-generated for any prime power $q$
and any integer $n\geq 3$, except when $(n,q)\in\{(3,2),(3,3),(3,5),(4,2),(4,3),(5,2)\}$.
\end{theorem}

The groups $\SU_{n}(q^2)$ with $n\leq 7$ have been studied in \cite{SU7,SU6,SU35,SU4}: so here we 
assume $n\geq 8$. Actually, it was proved by Gavioli, Tamburini and Wilson  that  
$\SU_{2n}(q^2)$ is $(2,3)$-generated for all $q$ and all $n\geq 37$ and that
$\SU_{2n+1}(q^2)$ is $(2,3)$-generated for all odd $q$ and all $n\geq 49$ \cite{TW,TWG}. 
Their bounds on the rank were slightly improved in \cite{ST}. 
We point out that the problem of the $(2,3)$-generation of the following groups is still open:
\begin{itemize}
\item $\PSp_{2n}(q)$, where $4\leq n\leq 24$ and $q$  is odd;
\item $\Omega_{2n+1}(q)$ and $\P\Omega_{2n}^+(q)$, where  $4\leq n\leq 40$, any $q$;
\item $\P\Omega_{2n}^-(q)$, where  $4\leq n\leq 43$ if $q$ is odd and  $n\geq 4$ if $q$ is even.
\end{itemize}

Following a successful idea introduced in \cite{T87}, 
our $(2,3)$-generators $x,y$ are very close to being permutations and tend to be uniform with respect to $n$ and $q$.
They are described in Section \ref{gene}, where we also give some preliminary results.
The proof of Theorem \ref{main} uses the knowledge of the maximal subgroups of a classical group. 
An important step is proving the existence of a bireflection, that is, of an element with a fixed points space of 
dimension $n-2$. This allows us to use the following result due to Guralnick and Saxl.

\begin{theorem}\cite[Theorem 7.1]{GS} \label{GS8} 
Let $V$ be a finite dimensional vector space of dimension $n\geq 9$ over an algebraically closed field $\F$ of
characteristic $p\geq 0$. Let $G$ be a finite irreducible subgroup of $\GL(V)$ which is primitive and
tensor-indecomposable on $V$. 
Define $\nu_G(V)$ to be the minimum dimension of $(\lambda g-1)V$ for $g\in G$, $\lambda$ a scalar with $\lambda g\neq
1$.
Then either $\nu_G(V)> \max\left\{2,\frac{\sqrt{n}}{2}\right\}$ or one of the following holds:
\begin{itemize}
\item[\rm{(a)}] $G$ is classical in a natural representation;
\item[\rm{(b)}] $G$ is alternating or symmetric of degree $c\geq 7$ and $V$ is the deleted permutation module of
dimension $c-1$ or $c-2$;
\item[\rm{(c)}] $G$ normalizes the group $\PSU_5(4)$ with $p\neq 2$ and $n=10$.
\end{itemize}
\end{theorem}

In Section \ref{action} we analyse the characteristic polynomial of the commutator $[x,y]$: its roots play an important 
role, e.g, for the irreducibility of $\langle x,y\rangle$ and  the application of  Theorem \ref{GS8}. The 
$(2,3)$-generation of $\SU_{n}(q^2)$ for $n\neq 8,11$  is proved in Section \ref{non
3}.
Unfortunately, the cases  $q\in \{ 2,3, 5,7,8,11\}$ require  an ad-hoc analysis, see Section \ref{eccez}.
Finally, in Sections \ref{8} and \ref{11} we consider the cases $n=8,11$, respectively.

\section{Generators and preliminary results}\label{gene}

Let $q=p^f$, where $p$ is a prime. Suppose $n=3m+r\geq 8$ with $r=0,1,2$ and denote
by $\C=\{e_1,\ldots,e_n\}$ the canonical basis of $V=\F^n$, where $\F$ is the algebraic closure of $\F_p$.
Denote by $\omega\in \F$  a primitive cubic root of unity.
Making  $\GL_n(q^2)$ act on the left on $V$, we define our generators $x_n=x_n(a),y_n=y_n(a)$ of  
respective orders $2$ and $3$, via their action on $\C$.
For $q>2$  they depend on the parameter $a\in \F_{q^2}$, for which we will always assume that
\begin{align}
\label{1} \left\{\begin{array}{c}
\F_{q^2}=\F_p[a^3], \quad c:=a^{q+1}-4\neq 0 \equad \\ 
\gamma:=a^3+a^{3q}-6a^{q+1}+8 \neq 0 
\end{array}\right. & \textrm{ if } p\neq 3,\\
\label{3} \F_{q^2}=\F_3[a]\equad  a^{q}+a^{q-1}+1=0\qquad  &  \textrm{ if } p= 3.
\end{align}
We will show that there always exists $a$ satisfying the above conditions and further ones which are sufficient 
to guarantee
that  $\langle (-1)^n x_n(a),y_n(a)\rangle=\SU_n(q^2)$
(see Propositions \ref{car ne 3} and \ref{policar3}).

The matrices $x_n=x_n(a),y_n=y_n(a)$ act on $\C$ as 
follows:
\begin{itemize}
\item if $r=0$ and $n>9$ then $x_n$ fixes both $e_1$ and $e_2$;
if $r=1$ then $x_n$ swaps $e_1$ and $e_2$ and, if $(n,q)\neq (10,2)$, $x_n$ fixes $e_3$;
if $r=2$ then $x_n $ swaps $e_1$ with $ e_3$ and $ e_2$ with $ e_4$;
\item if  $n\neq 8,11$ then $x_n$ fixes both $ e_{n-7}$ and $e_{n-4}$ when $q>2$
and it swaps $ e_{n-7}$ and $e_{n-4}$ when $q=2$;
\item if $n=9$ then $x_9$ fixes $e_1$; if $n=11$ then $x_{11}$ fixes $e_7$;
\item $x_n$ fixes $e_{3j+5+r}$ for any $j=0,\ldots,m-5$;
\item $x_n$ swaps $e_{3j+3+r}$ and $e_{3j+4+r}$ for any $j=0,\ldots,m-3$;
\item  $x_n$ acts on  $\langle e_{n-3},e_{n-2},e_{n-1},e_n\rangle$ as the matrix
 $$
 \begin{pmatrix}
0 & 1& 0 & 0\\
1 & 0 & 0& 0\\
0 & 0 & -1 & a\\
0 & 0 & 0 & 1
      \end{pmatrix} \; \textrm{ if } q>2 \textrm{ and } p\neq 3,$$
$$ \begin{pmatrix}
0 & 1& 0 & 0\\
1 & 0 & 0& 0\\
0 & 0 & 0 & 1\\
0 & 0 & 1 & 0
      \end{pmatrix} \; \textrm{ if }  p= 3,\quad \quad
      \begin{pmatrix}
  1  &  1  & \omega   & 0 \\
  1  &  0  & \omega &  \omega \\
 \omega^2  & \omega^2 &    0 & 1 \\
 0 & \omega^2 &  1 &   1 
      \end{pmatrix}\;  \textrm{ if }  q=2;$$
\item  $y_n$ acts on $\langle e_1,\ldots,e_{n-3}\rangle$ as the permutation
$$\prod_{j=0}^{m-2} \left(e_{3j+1+r},e_{3j+2+r},e_{3j+3+r}\right);$$
\item  $y_n$ acts on  $\langle e_{n-2},e_{n-1},e_n\rangle$ as the matrix
 $$ \begin{pmatrix}
bc^{-1} & -a^q\gamma c^{-2} & 2\gamma c^{-2}\\
1 & -bc^{-1} & -b^qc^{-1}\\
0 & 1 & 0
      \end{pmatrix} \; \textrm{ if } q>2 \textrm{ and } p\neq 3,$$
$$ \begin{pmatrix}
1 & 0 & a\\
-a^q & 1 & a\\
0 & 0 & 1
      \end{pmatrix} \; \textrm{ if }  p= 3,\quad\quad
      \begin{pmatrix}
0 & 0 & 1\\
1 & 0 & 0\\
0 & 1 & 0
      \end{pmatrix} \; \textrm{ if }  q=2,$$
where $b:=2a-a^{2q}$.
\end{itemize}
\noindent Note that $\gamma=-(a^q b+a b^q+2c)=\gamma^q$.

For $q>2$ the similarity invariants (i.e., the nontrivial invariant factors) of $x_n$ and $y_n$ are, respectively,
\begin{equation}\label{sim_x}
d_1(t)=\ldots=d_{m-r}(t)=t-1,\quad d_{m-r+1}(t)=\ldots =d_{2m}(t)=t^2-1
\end{equation}
and
\begin{equation}\label{sim_y}
d_{1}(t)=\ldots=d_r(t)=t-1,\quad d_{r+1}(t)=\ldots = d_{m+r}(t)=t^3-1.
\end{equation}
Notice that $\det(y_n)=1$ and that $\det(x_n)=(-1)^n$. More precisely, 
$\det(x_n)=-1$ when $m$ is odd and  $r=0,2$, or when $m$ is even and $r=1$.

Now, recalling the definition of $c$ and $\gamma$ given in \eqref{1}, let
$$J_n=\left\{ \begin{array}{ll}
\I_n  & \textrm{ if } q=2,\\
\diag(c\I_{n-2}, J_2) & \textrm{ if } p\neq 3 \textrm{ and } q>2,\\
\diag(\I_{n-2}, J_2) & \textrm{ if } p = 3,
\end{array}\right.$$
where $ J_2=\gamma c^{-1} \begin{pmatrix}
2 & -a^q \\
-a & 2
\end{pmatrix}$ if $p\neq 3$ and $J_2=\begin{pmatrix}
0 & 1 \\
1 & 0
\end{pmatrix}$ if $p=3$.

Clearly $J_n^T=J_n^\psi$, where $\psi: \GL_n(q^2)\rightarrow \GL_n(q^2)$ is defined by $\psi((a_{i,j}))=
( a_{i, j}^q) $.
Also, $\det(J_n)=-c^{n-3}\gamma^2$ if $p\neq 3$ and $q>2$, and $\det(J_n)=-1$ if $p=3$.
Since
$$x_n^T J_n x_n^\psi = J_n \equad y_n^T J_n y_n^\psi = J_n,$$
it follows that
$\langle x_n,y_n\rangle\leq \U_n(q^2)$ and $\langle (-1)^n x_n, y_n\rangle\leq \SU_n(q^2)$.
We will use the obvious fact that, if  the projective image of $\langle x_n,y_n\rangle$ is $\PSU_n(q^2)$,
then $\langle (-1)^n x_n,y_n\rangle=\SU_n(q^2)$ with $(-1)^n x_n$ again an involution.

From now on, we set 
$$H=\langle (-1)^n x_n, y_n \rangle.$$
and we simply write $x,y$ for $x_n,y_n$.

\begin{lemma}\label{tracce}
Let $q>2$. For  $p\neq 3$ assume that $b+a c = a^{q+2}-a^{2q}-2a \neq 0$. 
Then, $a$ belongs to the subfield generated by the traces of the elements of $H$.
Thus, $H$ is not conjugate to a subgroup of $\GL_n(q_0)$ for any $q_0<q$ by the assumption 
$\F_p[a^3]=\F_{q^2}$.
\end{lemma}

\begin{proof}
Suppose first that $p\neq 3$ and consider the traces of the following elements:
$$\begin{array}{rclcrcl}
\tr(xy) & =& a+\frac{b}{c}, & \quad & \tr(yxy) & = &  a^q+\frac{b^q}{c}\neq 0, \\
\tr((yxy)^2) & =& \frac{b^{2q}-(c+2)bc}{c^2}, & \quad & 
\tr((yxy)^3) & =& \frac{b^{2q}-(c+1)bc}{c^2}\cdot \tr(yxy).\\
\end{array}$$
Note that these values do not depend on $n$.
It follows that
$$\tr(xy)+\tr((yxy)^2)-\frac{\tr((yxy)^3)}{\tr(yxy)}=a.$$
Finally, if $p=3$, then $\tr(xy)=a$.
\end{proof}

\begin{lemma}\label{Nbeta}
Recall  $q=p^f$. If $f$ is not a $2$-power, write $f=\bar f\ell$, where $\ell$ is the smallest odd prime dividing $f$.
Given the polynomials
$$g_1(t)=t^{q+1}-\kappa\;\; \;(\kappa \in \F_q^*) \equad
g_2(t)=t^{q+1}+t^{q}+t,$$
call $\Nb_1$ the number of roots $\alpha$ of $g_1(t)$ such that $\F_p[\alpha^3]=\F_{q^2}$,  
$\Nb_2$ the number of roots $\alpha$ of $g_2(t)$ such that $\F_p[\alpha]=\F_{q^2}$. Then:
\begin{center}
\begin{tabular}{|c|c|c|c|}\hline
& if  $f = 2^s$, $s\geq 0$ & if $f$  is an odd prime & otherwise\\\hline &&&\\[-10pt]
$\Nb_1\geq $ & $q-5$ & $q-3p-8$ & $p^{3\bar f-1} $\\
$\Nb_2\geq $ & $q-1$ & $q-p$ & $p^{3\bar f-1}$\\\hline
\end{tabular}
\end{center}
\end{lemma}

\begin{proof}
Since the norm function $\F_{q^2}^\ast \to \F_{q}^\ast$ is an epimorphism, $g_1(t)$ has $q+1$ distinct roots in $\F_{q^2}^\ast$.
Clearly $g_2(t)$ is separable: we show that all its roots belong to $\F_{q^2}$.
To this purpose, let $\alpha$ be a root of $g_2(t)$. Then 
(\emph{i}) $\alpha^{q}(\alpha+1)+\alpha=0$ and (\emph{ii}) $\alpha^q=-\alpha(\alpha^q+1)$. Raising (\emph{i}) to the $q$ we get 
$\alpha^q=-\alpha^{q^2}(\alpha^q+1)$ whence, using (\emph{ii}), $(\alpha^q+1)(\alpha^{q^2}-\alpha)=0$. 
Now, $\alpha^q=-1$ gives the contradiction $-\alpha-1+\alpha=0$. We conclude that $\alpha^{q^2}=\alpha$.

Calling $M_1$ the number of roots $\alpha$ of $g_1(t)$ such that $\F_p[\alpha^3]\ne \F_{q^2}$ and 
$M_2$ the number of roots $\alpha$ of $g_2(t)$ such that $\F_p[\alpha]\ne \F_{q^2}$, 
we have 
$$\Nb_i= q+1 -M_i .$$
Note that the roots $\alpha$ of $g_1(t)$ such that $\alpha^3\in \F_q$ are roots of $t^6-\kappa^3$:
so they are at most $6$. 
Similarly, $g_2(t)$ has $2$ roots in $\F_q$, namely $0,-2\in \F_p$.
In particular, when $f=2^s$, we have $M_1\le 6$ and $M_2\le 2$ as $\F_q$ is the unique maximal subfield of $\F_{q^2}$.
\smallskip

\noindent \textbf{Case 1.} $f=2k+1$ an odd prime.
For $g_1(t)$ we need to exclude also the roots $\alpha$ such that $\alpha^3\in \F_{p^2}$, i.e., those such that
$\alpha^{3p^2}=\alpha^3$.
It follows  $\alpha^{3q}=\alpha^{3p^{2k+1}}= \alpha^{3p^{2k}p}= \alpha^{3p}$. 
So these roots satisfy the equation $t^{3p+3}=\kappa^3$. We deduce $M_1\leq 6+ 3p+3=3p+9$.
For $g_2(t)$ we need to exclude also the roots $\alpha\in \F_{p^2}$, i.e., those for which
$\alpha^{p^2}=\alpha$. It follows $\alpha^q= \alpha^{p^{2k+1}}= \alpha^p$. 
So these roots satisfy  the equation $t^{p+1}+t^{p}+t=0$. 
Thus $M_2\leq p$, since $0,-2\in \F_p$.\\
\noindent \textbf{Case 2.} $f = \ell \bar f$,  where $\ell$ is as in the statement and $\bar f>1$.
The elements $\beta\in \F_{q^2}\setminus \F_q$ such that $\F_p[\beta]\ne \F_{q^2}$ lie in subfields 
of order $p^s$ with $s\le {2\bar f}$. Hence they are at most:
$$ M=\sum_{i=1}^{2 \bar f}p^i=p+p^2+\dots +p^{2\bar f}
=\frac{p}{p-1}\left(p^{2\bar f}-1\right)\leq  2 p^{2\bar f}.$$

For each $\beta \in \F_{q^2}$,
there are at most $3$ roots $\alpha$ of $g_1(t)$ such that $\alpha^3=\beta$.
Hence there are at most $3M$ roots $\alpha$ of $g_1(t)$ such that $\F_p[\alpha^3]=\F_{q_0}\neq \F_q, \F_{q^2}$.
If $(p,\bar f)\not \in \{(2,2),(2,3),(3,2) \}$, we obtain
$$\Nb_1\ge q-5- 6 p^{2\bar f} \geq p^{3\bar f} -6 p^{2\bar f}-5
=p^{2\bar f} (p^{\bar f}-6)-5\geq p^{3 \bar f-1}.$$
If $(p, \bar f)=(2,2)$, then $M=30$ and so $\Nb_1\geq 2^{2\ell}-95$; if $(p, \bar f)=(2,3)$, then $M=126$, 
whence $\Nb_1\geq 2^{3\ell}-383$. In both cases, if $\ell\neq 3$, then $\Nb_1\geq 2^{3\bar f-1}$.
If $(p,\bar f)=(3,2)$, then $M\leq 120$ and so again $\Nb_1\geq 3^{2\ell}-365\geq 3^6-365\geq 3^{3\bar f-1}$.
When $q=64, 512$, a direct computation shows that $\Nb_1\geq  60$ and $\Nb_1\geq 486$, respectively.

Recalling that the roots of $g_2(t)$ which are in $\F_q$ are also in $\F_p$ we have:
$$\Nb_2\ge q + 1-2 p^{2\bar f}\geq p^{2\bar f} (p^{\bar f}-2)+1\ge p^{3\bar f-1}.$$
\end{proof}

\begin{lemma}\label{ei}
Suppose $q>2$. If $W$ is an $\langle x,y\rangle$- or an $\langle x,y\rangle^T$-invariant subspace of $\F^n$ 
such that $e_{n-1}\in W$ or $e_n\in W$, then $W=\F^n$.
\end{lemma}

\begin{proof}
The subspace $\left\langle e_{n-2},e_{n-1},e_{n} \right\rangle$ is $y$-invariant and $y^T$-invariant.

If  $p\neq 3$, the $3\times 3$ matrix whose columns are the last three components of $e_j, ye_j, 
y^2e_j$ has  determinant $(a^{3q}-8)\gamma c^{-3}$ or  $(a^3-8)\gamma c^{-3}$ 
according to $j=n-1$ or $j=n$. Similarly,
the matrix whose columns are the last three components of $e_j, y^Te_j, \left(y^T\right)^2e_j$ 
has determinant $-1$  for $j=n-1, n$. It follows that $e_{n-2}\in W$.

If $p=3$, then  $x e_n = x^T e_n=e_{n-1}$ and $a^{-1}(ye_n -e_n) - x e_n = -a^{-q} (y^T e_{n-1}-e_{n-1}-ax^T 
e_{n-1})=e_{n-2}$.
Again, we obtain  that $e_{n-2}\in W$. 

By induction, one can see that, for any $j\leq n-3$, there exists an element $g_j$ of $\langle x,y\rangle$ 
such that $g_je_{n-2}=e_j$. The same holds also for $\langle x,y\rangle^T$.
We conclude that $\C=\{e_1,\ldots,e_n\}\subseteq W$, whence our claim.
\end{proof}

Given $g\in \Mat_n(\F)$, $\lambda\in \F$ and a $g$-invariant subspace $W$ of $\F^n$, define
$$W_\lambda(g)=\left\{w\in W \mid gw=\lambda w\right\}.$$

\begin{lemma}\label{rem d}
Let $U$ be a $G$-invariant subspace of $V=\F^n$, where $G\le \GL_n(\F)$. 
\begin{itemize}
 \item[(1)]  There exists a complement $\overline{U}$ of $U$ which is $G^T$-invariant;
 \item[(2)]  for each  $g\in G$ and any $\lambda\in \F$,
$$\dim V_\lambda(g) \le \dim U_\lambda (g)+\dim  \overline U_\lambda (g^T).$$
\end{itemize}
In particular, if $\sigma$ is a simple eigenvalue of $g\in G$ (i.e., $\sigma$ is a root of multiplicity $1$ of the
characteristic polynomial of $g$), then either
\begin{itemize}
 \item[(i)] $g_{|U}$ has the eigenvalue $\sigma$; or
 \item[(ii)] $g^T_{|\overline{U}}$ has the eigenvalue $\sigma$.
\end{itemize}
\end{lemma}

\begin{proof} 
(1) Choose a basis  $\B=\{u_1, \dots , u_h, w_1, \dots , w_{n-h}\}$ of $\F^n$, where $\{u_1, \dots , u_h\}$ is a basis 
of
$U$ and consider the matrix $P$ whose columns are the vectors
of $\B$. 
Then  $P^{-1}GP$ and its transpose  $P^TG^T P^{-T}$ consist of matrices of respective shapes:
\begin{equation}\label{G}
\overline g=\begin{pmatrix}
A&B\\
0&C
\end{pmatrix},\qquad \overline g^T=\begin{pmatrix}
{A}^T&0\\
{B}^T&{C}^T
\end{pmatrix}.
\end{equation}
It easily  follows that the subspace $\overline U= \left\langle P^{-T}e_{h+1}, \dots , P^{-T}e_{n}\right\rangle$
is $G^T$-invariant.\\
\noindent (2) Substituting $G$ with $P^{-1}GP$ we may suppose that $U=\langle e_1, \ldots, e_h\rangle$, 
$\overline U=\langle e_{h+1}, \ldots, e_n\rangle$ and  
$g=\overline g$ as in \eqref{G}. Call  $\pi: \F^n\to \F^{n-h}$ the projection onto the last $n-h$ coordinates.
Then $\pi \left(V_\lambda (g)\right)$ coincides with the subspace $W$ of $\F^{n-h}$ defined by:
$$W= \left\{w\in\F^{n-h}\mid Cw=\lambda w, \ Bw\in \im (\lambda I_h-A)\right\}.$$
Moreover $\Ker( \pi)\cap V_\lambda(g)=U_\lambda (g)$.
Since $C$ and $C^T$ are conjugate, we have 
$\dim W\leq \dim \{w\in\F^{n-h}\mid Cw=\lambda w \}=\dim \{w\in\F^{n-h}\mid C^T w=\lambda w \}=
\dim \overline U_\lambda (g^T)$. It follows 
$\dim V_\lambda(g)=\dim U_\lambda(g)+\dim W\leq \dim U_\lambda(g)+\dim \overline U_\lambda (g^T).$
\end{proof}

Given a finite group $G$, let $\varpi(G)$ be the set of the prime divisors of $|G|$.
For simplicity, if $g \in G$,  we write $\varpi(g)$ for $\varpi(\langle g \rangle)$.
In Proposition \ref{23578} we use the following result due to Liebeck, Praeger
and Saxl. Namely, we construct suitable elements $g_i\in H$ such that
$\cup\varpi(g_i)=\varpi(\SU_n(q^2))$ and conclude that $H=\SU_n(q^2)$. 
For the order of the unitary groups, one can refer to \cite{Ho}.

\begin{lemma}\label{LPS}
Let $M$ be a subgroup of $G=\PSU_n(q^2)$, $n\geq 7$.
If $\varpi(M)=\varpi(G)$, then $M=G$.
\end{lemma}

\begin{proof}
It follows from \cite[Corollary 5 and Table 10.7]{LPS}.
\end{proof}

\section{Action of the commutator for  $n\neq 8,11$ and $q>2$}\label{action}

For $q>2$ consider the commutator $[x,y]=x^{-1}y^{-1}xy$ and  set
$$\S=\left\{e_{n-7}, e_{n-4},e_{n-3},e_{n-2},e_{n-1},e_{n}\right\}.$$
If $xe_{n-7}=e_{n-7}$, which happens precisely when $n\neq 8,11$, then $[x,y]e_{n-3}=e_{n-7}$. In this case, the subset 
$\C\setminus \S$ and the subspace $\langle \S \rangle$ are 
$[x,y]$-invariant. 
This is why, in this section, we always assume that $n=9,10$ or $n\geq 12$.

The commutator  $[x,y]$  acts on  $\C \setminus \S $ as the  permutation:
$$\begin{array}{ll}
  (e_3,e_4)\prod\limits_{i=0}^{\frac{n-12}{3}}(e_{2+3i}, e_{7+3i}, e_{6+3i}), & \textrm{ if } r=0 \textrm{ and } n\geq 
9;\\
 (e_1, e_5, e_4, e_2) \prod\limits_{i=0}^{\frac{n-13}{3}}(e_{3+3i}, e_{8+3i}, e_{7+3i}),& \textrm{ if } r=1 \textrm{ 
and }    
n\geq 10;\\
(e_1, e_6, e_5, e_3, e_4, e_9, e_8, e_2)\prod\limits_{i=1}^{\frac{n-14}{3}}(e_{4+3i}, e_{9+3i}, e_{8+3i}), & \textrm{ 
if } r=2 
\textrm{ and } n\geq 14.
\end{array}$$
It follows that  the restriction of $[x,y]$ to
$\langle \C\setminus \S\rangle$ has order 
$$\left\{\begin{array}{cl}
2 & \textrm{ if } n=9;\\
4 & \textrm{ if } n=10;\\
8 & \textrm{ if } n=14;\\
3\cdot 2^{r+1} & \textrm{ if } n=12,13 \textrm{ or } n\geq 15.
\end{array}\right.$$

The characteristic polynomial of the restriction of $[x,y]$ to $\langle \S \rangle$ is
$\chi(t)=(t-1)(t+1)\chi_0(t)$ with
\begin{equation}\label{chi0}
\chi_0(t)=\left\{ \begin{array}{ll}
t^4 +\frac{(c+2)(\gamma-c^2)}{c^2} t^3+\frac{\gamma+2c}{c}t^2+\frac{(c+2)(\gamma-c^2)}{c^2} t+1 & \textrm{ if } p\neq 
3,\\[3pt]
(t-1)^2 (t^2 - a^{q+1} t +1) & \textrm{ if  } p=3.
\end{array}\right.
\end{equation}
By the previous discussion, the $24$-th power of the commutator
$$C=[x,y]^{24}$$
has a fixed points space of codimension $\leq 4$. Moreover, for  a fixed $a$, its order does not
depend on $n$. When possible, we impose conditions on $a$ that make $C$  a bireflection.

\begin{lemma} \label{poli}
Suppose $p\neq 3$ and $a^{q+1}=3$. Then $\gamma=a^3- 10 +\frac{ 27}{a^3}$ and
$$\chi_0(t)=(t+\omega)(t+\omega^{-1})(t^2 + \gamma t + 1).$$
The order of the roots of  $t^2 + \gamma t + 1$ does not divide $48$ whenever
$$\begin{array}{clccl}
(1) &  \gamma+j\neq 0,\;\; j\in \{0, \pm 1, \pm 2\}; & \quad & (2) &  \gamma^2-j\neq 0, \;\; j=2,3;\\
(3) &  \gamma^4-4\gamma^2+j\neq 0,\;\;  j=1,2; & \quad & (4) &  \gamma^8-8\gamma^6+20\gamma^4-16\gamma^2+1\neq 0.
\end{array}$$
Under such conditions, there are exactly two eigenvalues of $[x,y]$ whose order does not divide $48$:
they are
$$\sigma^{\pm 1}=a^{\pm 3}\quad \textrm{ if } p=2 \equad \sigma^{\pm 1} = \frac{-\gamma\pm \sqrt{\gamma^2-4}}{2}\quad 
\textrm{ if } p\geq 5.$$
\end{lemma}

\begin{proof} 
The factorization of $\chi_0(t)$ is routine. While, over $\Q$, we have:
\begin{eqnarray*}
t^{48}-1 & =& (t^8+t^6-t^2-1)(t^8-t^6+2t^4-t^2+1)\cdot\\
& & \cdot (t^{16}-t^{12}+2t^8-t^4+1)(t^{16}-t^8+1). 
\end{eqnarray*}
The conditions are obtained calculating $\XGCD\left( p_i(t),t^2+\gamma t + 1\right)$
in MAGMA \cite{MAG}, where $p_i(t)$ runs among the factors of $t^{48}-1$.
Conditions  (1) are obtained considering the irreducible factors $p_i(t)$ of $t^8+t^6-t^2-1$.
Conditions (2)--(4) are obtained in a similar way considering the other irreducible factors.

The fact that $\sigma\neq \sigma^{-1}$ follows from conditions (1). Namely, if $p=2$ these conditions imply  
$\gamma=a^3+a^{-3}\neq 0$ and, if $p\geq 5$, they imply $\gamma \neq \pm 2$.
\end{proof}

\begin{proposition}\label{car ne 3} 
If $p\neq 3$ and $q\neq 2,5,7,8,11$, there exists $a$ satisfying the following 
condition, which includes \eqref{1}:
\begin{equation}\label{cond2}
\begin{array}{l}
\F_{q^2}=\F_p[a^3], \quad a^{q+1}=3 \textrm{ and the roots of } \\
t^2+\gamma t + 1 \textrm{ have order not dividing } 48.
\end{array}
\end{equation}
\end{proposition}

\begin{proof}
By Lemma \ref{poli} it suffices to check that $\Nb_1> 6\cdot 5 +12\cdot 2 + 24 \cdot 2+48= 150$. 
We apply Lemma \ref{Nbeta} taking $g_1(t)=t^{q+1}-3$.
If $f=2^s$, $s\geq 0$, we have $\Nb_1 \geq q-5$ and $q-5 > 150$ for all $q> 155$.
If $f$ is an odd prime, we have $\Nb_1\geq q-3p-8$ and $ p^3-3p-8 > 150$ for all $p\geq 7$.
For the other values of $f$, we have  $\Nb_1\geq p^5$ and $p^5>150$ for all $p\geq 5$.
So, we are left to check the following cases:
$$\begin{array}{clccl}
(1) &  p=2 \textrm{ and  } 1 \leq f \leq 7; & \quad &  (2) & p=5 \textrm{ and  }1\leq f \leq 3;\\
(3) & p=7,11 \textrm{ and }f=1,2; & \quad & (4) & 13\leq q=p\leq 151.
\end{array}$$
Note that for $p=2$, conditions (1)--(4) of Lemma \ref{poli} simply become 
$a^9+1\neq 0$. Hence, if $q=2^f$ with $f>3$, our claim is true as  $\Nb_1 \geq  11>9$.
For $q\in \{ 4, 25, 49, 121, 125 \}$ take $a\in \F_{q^2}$ according to Table \ref{tab2}.
If $q=p \in \{ 17, 23, 29, 41, 47,$ $53, 59, 71, 83, 89, 101, 107, 113, 131,137, 149  \}$,
take $a=\sqrt{-3}$; if $q=p\in \{ 13, 43,61,$ $73, 79, 109, 127, 139, 151\}$, take
$a=\frac{1+\sqrt{-11}}{2}$;
if $q=p\in \{ 37, 103\}$, take
$a=-3+\sqrt{6}$; if $q=p\in \{ 19, 31, 67, 97 \}$, take
$a=\frac{-5+\sqrt{13}}{2}$.
\end{proof}

\begin{table}[ht]
$$\begin{array}{|cc|cc|cc|}\hline
q & m(t) & q & m(t) & q & m(t) \\\hline
4 & t^4+t^3+t^2+t+1 &
25 & t^4+2t^3+t-1  & 
49 & t^4+3 t^3+2 t+2 \\
121 & t^4-t^3-t^2-3t-2 & 
125 & t^6-t^5+t+2  & &\\ \hline
  \end{array}$$
\caption{Minimal polynomial $m(t)$ of $a$ over $\F_p$.}\label{tab2}
\end{table}

We now consider the case $p=3$.

\begin{proposition} \label{policar3}
Suppose $p= 3$. Then there exists $a$  satisfying \eqref{3}.
For any such $a$, provided that $q\neq 3$ and 
$a^4 -a^3 + a^2 + a-1=0$ when $q=9$, 
the factor $\chi_1(t)=t^2-a^{q+1} t+1$  of $\chi_0(t)$ defined in 
\eqref{chi0} has two distinct roots 
$$\sigma^{\pm 1}= \frac{a^2\pm (a-1)\sqrt{a^2-a-1}}{a+1}$$
whose order does not divide $48$.
 \end{proposition}

\begin{proof} 
The polynomial $\chi_1(t)$ has two distinct roots unless $a^2-a-1=0$, in which case $a \in \F_{3^2}$.
Now we impose that $\sigma^{16}\neq 1$.
Writing $t^{16}-1=(t^8-1)(t^8+1)$, we get: 
$$\begin{array}{lclcl}
\gcd(\chi_1(t), t^{8}-1)\ne 1& \Longleftrightarrow & a^4+a^2-a+1=0;\\
\gcd(\chi_1(t), t^{8}+1)\ne 1 & \Longleftrightarrow & (a^4+a^3-1)(a^4-a^3+a+1) = 0.
\end{array}$$
In both cases, it follows that $a\in \F_{9^2}$: it suffices to take $a$ as in the statement.
\end{proof}

\section{The case  $n\neq 8,11$ and $q\ne 2,3,5,7,8,11$}\label{non 3}

Throughout this section, we suppose that  $q\ne 2,3,5,7,8,11$ and $n=9,10$ or $n\geq 12$.
Moreover we assume that $a$ satisfies conditions \eqref{cond2} when $p\neq 3$, 
conditions \eqref{3} when $p=3$ and the further condition $a^4-a^3+a^2+a-1=0$ when $q=9$.

By the results of Section \ref{action}, the subspaces  $\langle \S \rangle$ and $\langle \C\setminus \S \rangle$ are 
both $[x,y]$-invariant and $[x,y]^T$-invariant. 
Moreover, $[x,y]_{|\langle \C\setminus \S \rangle}$ has order 
dividing $24$.
On the other hand, $[x,y]_{|\langle \S \rangle}$ has 
characteristic polynomial:
$$
\chi(t)=\left\{
\begin{array}{ll}
 (t-1)(t+1)(t+\omega)(t+\omega^{-1})(t-\sigma)(t-\sigma^{-1}) &  \textrm{ if } p\neq 3,\\
(t-1)^3(t+1)(t-\sigma)(t-\sigma^{-1})  & \textrm{ if } p = 3.
\end{array}\right.
$$
Recall that $\sigma^{24}\neq \pm 1$. 
The fixed points space of $[x,y]_{|\langle \S \rangle}$ is generated by
\begin{equation}\label{v}
v=\left\{\begin{array}{ll}
(\zeri{n-9}, 0, a^3-6, 0,0, a, a^3-6, a^3-6, 2a^2, 4a) & \textrm{ if } p\neq 3,\\
(\zeri{n-9}, 0, a+1, 0,0, 1, a+1, a+1, -1, -1) & \textrm{ if } p=3,
\end{array}\right.
\end{equation}
where $\zeri{i}$ means a sequence of $i$ zeros.
Thus, denoting by $\Jor(i)$ a unipotent Jordan block of size $i$, the Jordan form 
of  $[x,y]_{|\langle \S \rangle}$ is:
\begin{equation}\label{Jordan}
\left\{\begin{array}{ll}
\diag(1,-1,-\omega,-\omega^{-1}, \sigma, \sigma^{-1}) & \textrm{ if } p \geq 5,\\
\diag(\Jor(2) , \omega, \omega^{-1}, \sigma, \sigma^{-1}) & \textrm{ if } p =2,\\
\diag(\Jor(3), -1, \sigma, \sigma^{-1}) & \textrm{ if } p =3.
\end{array}\right.
\end{equation}
It follows that $C=[x,y]^{24}$ is a bireflection and that $C$ and $C^T$ are diagonalizable.

We now determine some relevant eigenvectors of $[x,y]_{|\langle \S \rangle}$ and of its transpose.

\noindent \textbf{Case $p\neq 3$.} Taking $b=2a-a^{2q},c, \gamma$ as in \eqref{1}, from conditions  \eqref{cond2}, we get
$$b=\frac{2a^3-9}{a^2}, \quad b^q=\frac{6-a^3}{a}, \quad c=-1,\quad a^3 \gamma=a^6-10 a^3+27\neq 0.$$ 
Observe that $b\neq 0$, since $a^3-6=0$ contradicts the hypothesis $\F_p[a^3]=\F_{q^2}$.
The eigenspaces $\langle \S \rangle_{\sigma^{\pm 1}}([x,y])$ and $\langle \S \rangle_{\sigma^{\pm 1}}([x,y]^T)$ 
are generated, respectively, by:
$$\begin{array}{rcl}
s_{\sigma^{\pm 1}} &  = & 
\left(\zeri{n-9}, 0,1,0,0,(\gamma+1)b^{-q},\sigma^{\pm 1}, \sigma^{\mp 1},
-ab^{-q}(\sigma^{\pm 1}+1), (\gamma-2)b^{-q}\right)^T,\\[5pt]
\overline{s}_{\sigma^{\pm 1}}  &=&
\left(\zeri{n-9},0,1, 0,0,\frac{(\gamma+1)}{b}, \sigma^{\mp 1}, \sigma^{\pm 1},
\frac{\pm 3 \gamma(\sigma-\sigma^{-1})}{a b},\frac{\gamma (\sigma^{\mp 1}-2\sigma^{\pm 1}-1) }{ b}\right).
\end{array}$$

\noindent \textbf{Case $p = 3$.} 
The eigenspaces  $\langle \S \rangle_{\sigma^{\pm 1}}([x,y])$ and $\langle \S \rangle_{\sigma^{\pm 1}}([x,y]^T)$ 
are generated, respectively, by:
$$\begin{array}{rcl}
s_{\sigma^{\pm 1}}   & =&  
\left(\zeri{n-9}, 0,1,0,0,\frac{a^2+1}{a^2-1},\sigma^{\pm 1} , \sigma^{\mp 1},
\frac{-a \pm a(a-1)\sqrt{a^2-a-1}}{(a+1)^2(a-1)},\right.\\
& &\left.\frac{-a\mp a(a-1)\sqrt{a^2-a-1}}{(a+1)^2(a-1)} 
\right)^T,\\
\overline{s}_{\sigma^{\pm 1}}  &=&
\left( \zeri{n-9}, 0,1,0,0, -\frac{a^2+a-1}{a-1} ,\sigma^{\mp 1} , \sigma^{\pm 1},
\frac{ a(a+1)^2}{a-1} \mp a\sqrt{a^2-a-1}  , \right.\\
& &\left. \frac{a(a+1)^2}{a-1} \pm a\sqrt{a^2-a-1}\right)^T.
\end{array}$$

For any value of $p$, from $[x,y]^x=[x,y]^{-1}$ it follows $xs_{\sigma}=s_{\sigma^{-1}}$ and $x^T \overline{s}_{\sigma}=\overline{s}_{\sigma^{-1}}$. 

\begin{theorem}\label{Hirr}
The group $\langle x,y\rangle$ is absolutely irreducible.
\end{theorem}

\begin{proof}
Let $U$ be an $\langle x,y\rangle$-invariant subspace of $\F^n$. 
By Lemma \ref{rem d} there exists an $\langle x^T,y^T\rangle$-invariant complement $\overline{U}$.
Since $\sigma$ and $\sigma^{-1}$ are simple eigenvalues of $[x,y]$,  by the same lemma we may assume either $(i)$
$s_{\sigma} \in  U$ or $(ii)$ $\overline{s}_{\sigma} \in \overline{U}$.\\
\noindent \textbf{Case $(i)$.} From $xs_\sigma =s_{\sigma^{-1}}$ it
follows that $s_{\sigma},s_{\sigma^{-1}}\in U$. Hence $w \in \U$, where
\begin{equation}\label{w}
w=\left\{\begin{array}{ll}\frac{1}{\sqrt{\gamma^2-4}}\left(s_{\sigma}-s_{\sigma^{-1}}\right)= \left(\zeri{n-4}, 1, -1, 
-ab^{-q},  
0\right)^T \in U & \textrm{ if } p\neq 3,\\
\frac{\sigma}{\sigma^2-1}
\left(s_{\sigma}-s_{\sigma^{-1}}\right)= \left(\zeri{n-4}, 1, -1, \frac{a}{a^2-1}, \frac{-a}{a^2-1} \right)^T \in U
& \textrm{ if } p=3.
\end{array}\right.
\end{equation}
Direct calculation gives:
$$e_{n-1}=\frac{-b^q(a\gamma+b)}{a^2}w+\frac{a\gamma-b}{a^2}(y^2-xy^2)w,$$
if $p\neq 3$;
$$u:=\frac{(a-1)^2}{a^2}\left( y^2 w-x y^2 w+\frac{1}{a^2-1}w\right )=e_{n-1}-e_n$$
and
$$e_{n-1}=-\frac{a+1}{a^2}\left(u+yu+y^2u\right) $$
if $p=3$.
It follows that $e_{n-1}$  belongs to $U$ for all $p$, whence $U=\F^n$ by   Lemma \ref{ei}. \\
\textbf{Case $(ii)$.} From $x^T \overline{s}_{\sigma}=\overline{s}_{\sigma^{-1}}$ it follows that
$\overline{s}_{\sigma}, \overline{s}_{\sigma^{-1}}\in \overline{U}$. Hence, $\overline{w} \in \overline{U}$, where
\begin{equation*}\label{overline w}
\overline{w}=\left\{
\begin{array}{ll}
\frac{1}{\sqrt{\gamma^2-4}}\left(\overline s_{\sigma}-\overline s_{\sigma^{-1}}\right) = 
\left(\zeri{n-4}, -1, 1, 6\gamma (a b)^{-1},
 -3\gamma b^{-1}\right)^T & \textrm{ if } p\neq 3,\\
 \frac{\sigma}{1-\sigma^2}\left(\overline s_{\sigma}-\overline s_{\sigma^{-1}}\right) =
 \left(\zeri{n-4}, 1, -1, \frac{a^2+a}{a-1} , -\frac{a^2+a}{a-1}\right)^T& \textrm{ if } p= 3.
 \end{array}\right. 
\end{equation*}
For $p\ge 5$, we have:
$$\frac{(ab)^2}{18\gamma }\left(x^Ty^T-y^T\right)\overline 
w+\frac{\left(2a^6-24a^3+81\right)b}
{18a^2\gamma} \overline w  =\tilde u,$$
where $\tilde u=e_{n-1}-\frac{a}{2}e_n$, and
$$e_n =\frac{-2}{a^3-8}\left(a^2\tilde u+2ay^T\tilde u+4(y^T)^2\tilde u\right).$$
For $p=2$ we get: 
$$e_n =\frac{1}{a^8+a^2}\overline w +\frac{1}{a^6+1}\left(y^T+x^Ty^T\right)\overline w.$$
For $p=3$, setting
$$\tilde u =e_{n-1}-e_n = \frac{(a-1)^2}{a^2}\left( x^T y^T \overline w- y^T \overline w-\frac{(a+1)^2}{a-1}\overline 
w\right)$$
we get $\tilde u \in \overline{U}$ and 
$$e_n=\frac{a+1}{a^2}\left (\tilde{u}+y^T \tilde{u}+(y^T)^2 \tilde{u}\right).$$
From $e_n\in \overline{U}$ for all $p$ we get $\overline{U}=\F^n$ by   Lemma
\ref{ei}, whence $U=\{0\}$.
\end{proof}

\begin{lemma}\label{Cy}
The matrix $C$ does not centralize $\left \langle C^y, C^{y^2} \right \rangle$.
\end{lemma}

\begin{proof}
If the claim is false, then $C^{y^i}$ fixes $\langle s_\sigma\rangle$
and $\langle s_{\sigma^{-1}}\rangle$, namely $y^is_{\sigma}$ and $y^is_{\sigma^{-1}}$ are eigenvectors of $C$ for $i=1,2$. Note that $ys_{\sigma^{\pm 1}}=(\zeri{n-9}, 0, 0, 1, \ast)$ and 
$y^2 s_{\sigma^{\pm 1}}=(\zeri{n-9}, 1, 0, 0, \ast)$ does not belong to $\langle \S\rangle$.  It follows that $C(y^i s_{\sigma^{\pm 1}}) =y^i s_{\sigma^{\pm 1}}$. From here we get $C(y^i w)=y^i w$, $i=1,2$, with $w$ as in \eqref{w}.
Using $C^x=C^{-1}$ we deduce $C(xy^i w)=xy^i w$.
As $C$ fixes $e_{n-5}$ and $e_{n-6}$, it also fixes the vectors 
$v_1 = yw -e_{n-5}$, $v_2=y^2w$, $v_3=xyw-e_{n-6}$ and $v_4=xy^2w$, where, setting
$\beta=(a-1)(a+1)^2$:
$$\begin{array}{rcl}
 (a^3-6) v_1 &= & (\zeri{n-5}, 0, 0, -a(\gamma+1), a^3-3, a^2)^T,\\
a(a^3-6) v_2 &= & (\zeri{n-5}, a^4-6a, 0, a^6-8a^3+18 , 0, a^4-3a)^T,\\
(a^3-6) v_3 & =& (\zeri{n-5},0, -a(\gamma+1), 0, 3, a^2)^T, \\
a(a^3-6) v_4 &= & (\zeri{n-5}, a^4-6a, a^6-8a^3+18, 0, a^5-3a^2, 
a^4-3a) 
\end{array}$$
if $p\neq 3$ and
$$\begin{array}{rcl}
(a^2-1)v_1 & = & (\zeri{n-5}, 0,0,a^2+1, a^2-a,-a)^T,\\
\beta v_2 & = & (\zeri{n-5}, \beta, 0, a+1 , a^2 (a-1) , - a(a+1))^T,\\
(a^2-1) v_3 & =&  (\zeri{n-5}, 0, a^2+1, 0, -a, a^2-a)^T,\\
\beta v_4 & =&  (\zeri{n-5}, \beta, a+1, 0, -a(a+1), a^2(a-1))^T
\end{array}$$
if $p=3$.  Clearly $\mathcal{S}_0=\langle v_i\mid 1\leq i\leq 4\rangle$ is contained in $\langle \S\rangle$.
Considering coordinates of the $v_i$'s of position $n-4,n-3,n-1,n$, we see that these vectors are linearly independent
unless $p\geq 3$ and $2a^6 - 12a^3 + 27=0$ or $p=3$ and $a^4 + a^2 -a + 1=0$.
In the first case, raising to the $q$, one easily gets
$a^3=\frac{27}{4}=a^{3q}=\frac{27}{a^3}=4$, in contradiction with the hypothesis $\F_p[a^3]=\F_{q^2}$.
In the second case, one gets $q=9$, but for such value  of $q$ our $a$ is a root of 
$t^4 - t^3 + t^2 + t - 1$.

Thus $\mathcal{S}_0$ has dimension $4$. The vector $v$ defined in \eqref{v}
is fixed by $C$. Clearly, $v \not \in \mathcal{S}_0$, as its coordinate of position $n-7$ is nonzero. It follows that $C_{|\langle \S\rangle}$ fixes 
pointwise a $5$-dimensional space, an absurd, since from \eqref{Jordan} we see
that $C_{|\langle \S\rangle}$ has Jordan form $\diag(\rho,\rho^{-1},\I_4)$,
where $\rho=\sigma^{24}\neq \pm 1$.
\end{proof}

\begin{theorem}\label{prim}
The group  $\langle x,y\rangle$ is primitive.
\end{theorem}

\begin{proof}
By contradiction let 
$\F^n=T_1\oplus \dots \oplus T_\ell$, $\ell \ge 2$,
be a decomposition  preserved by $\langle x,y\rangle$.
By the irreducibility of $\langle x,y\rangle$ (Theorem \ref{Hirr}), the $T_j$'s must be permuted transitively, so their common dimension is $k$, say.
Recall that $C$ has Jordan form $\diag(\rho, \rho^{-1}, \I_{n-2})$, where $\rho=\sigma^{24}\neq \pm 1$.\\
\noindent \textbf{Case 1.}  $CT_j=T_j$  for all $1\leq j\leq \ell$. \\
In this case we may assume $s_{\sigma} \in T_1$ and either (\emph{i}) $s_{\sigma^{-1}} \in T_1$ or (\emph{ii})
$s_{\sigma^{-1}}\in T_2\neq T_1$.
If (\emph{i}) holds, from $xs_\sigma=s_{\sigma^{-1}}$, we have $xT_1=T_1$.
Hence, by the transitivity,  $T_1\ne yT_1=T_2$, say, and $yT_2=T_3$. 
The restriction of $y$ to $T_1\oplus T_2\oplus T_3$, for an appropriate basis, 
has shape 
$\left(\begin{smallmatrix}
0&0&\I_k\\
\I_k&0&0\\
0&\I_k&0
\end{smallmatrix}\right)$. 
Since $C_{|T_j}=\id$ for the $T_j$'s in the remaining orbits of $y$ (if any), $C^y$ and $C^{y^2}$ commute with $C$, against  Lemma \ref{Cy}.

If (\emph{ii}) holds then $[x,y]T_j=T_j$ for $j=1,2$ and,  from $xs_\sigma=s_{\sigma^{-1}}$ we have $xT_1=T_2$.  If  
$y$ fixes  $T_1$ or $T_2$, then it fixes both.
This gives $\ell =2$, 
$x$ conjugate to 
$\left(\begin{smallmatrix}
0&\I_k\\
\I_k&0
\end{smallmatrix}\right)$, in contrast with the similarity invariants \eqref{sim_x} of $x$. So $T_1$ and $T_2$
lie in orbits of length $3$. They cannot be the same: indeed from $yT_1=T_2$, say,  we get the contradiction $yT_2=T_1$.
We conclude that $T_1$ and $T_2$ are in different orbits of $y$ of length $3$. Again $C^y$ and $C^{y^2}$ commute with $C$, the same contradiction as above.\\
\noindent \textbf{Case 2.} $CT_1\neq T_1$.\\
Let $\Omega$ be the orbit of $T_1$ under $C$. The sum $W$ of the $T_j\in \Omega$ is a $C$-invariant subspace.
Setting $s=|\Omega|>1$, we have $C^sT_1=T_1$. Since $C^s$ is semisimple, $T_1$ has a basis $v_1,\ldots,v_k$ consisting of eigenvectors of $C^s$.
From $C^s v_i = \lambda_i v_i$, $1\leq i \leq k$, it follows that $C_{|W}$ has characteristic polynomial
$$\chi_W(t)=\prod_{i=1}^k (t^s-\lambda_i).$$
Now, $p$ cannot divide $s$, as it does not divide the order of $C$. Thus, $\gcd(p,s)=1$ and $t^s-\lambda_i$ has $s$ distinct roots.
If $\chi_W(t)$ does not have the roots $\rho^e$ for some $e=\pm 1$, it can only have the root $1$.
This gives $s=1$, a contradiction. So we may assume that $\rho^e$ is a root of $\chi_W(t)$ and $\lambda_1=\rho^{es}$.
Since $C^s$ has at most $3$ distinct eigenvalues, we have $s=2,3$.
If $s=2$, then $C$ acts as a $2$-cycle on the $T_j$'s, a contradiction since $C$ must act as an even permutation.
If $s=3$, then $\gcd(p,s)=1$ gives $p\neq 3$ and $\chi_W(t)$ has roots $\rho^e, \omega \rho^e, \omega^2 \rho^e$.
This implies that $\rho^e=\omega^{\pm 1}$, whence the order of $[x,y]$ is divisible by $9$.
Thus $\tau=[x,y]^8$ has order $9$. Since $C=\tau^3$ does not fix all $T_j$'s, it follows that $\tau$ has an orbit of length $9$ and 
$C$ has $3$ orbits of length $3$, a contradiction.
\end{proof}

\begin{lemma}\label{tens}
The group $\langle x,y\rangle$ is tensor-indecomposable.
\end{lemma}

\begin{proof}
Suppose that $V=V_1\otimes V_2$ is a tensor decomposition for the $\langle x,y\rangle$-module $V=\F^n$, where
$1<d_1=\dim V_1\leq d_2=\dim V_2$.
We can write $C=A\otimes B$, where we may assume
$A=\diag(\rho,\alpha_2,\ldots,\alpha_{d_1})$ and 
$B=\diag(1,\beta_2,\ldots,\beta_{d_2})$, as $C$ is diagonalizable.
Since the fixed points space of $C$ has dimension $n-2$, the only possibility is $A=\diag(\rho,1)$ and
$B=\diag(1,\rho^{-1})$, which is impossible since $n=d_1 d_2\geq 9$.
\end{proof}

\begin{proposition}\label{no C5}
The group $\langle x, y\rangle$ is not contained in any maximal subgroup $M$ of class $\Cl_5$ of $\U_n(q^2)$.
\end{proposition}

\begin{proof}
Suppose, by contradiction, that there exists $g \in \U_n(q^2)$ such that 
$$\langle x, y \rangle^g\leq  \GL_n(q_0)\left( \F_{q^2}^* \I_n\right) \leq \SU_n(q^2),$$
where $\F_{q_0}$ is a proper subfield of $\F_{q^2}$. Set
$$x^g = \vartheta_1 x_0,\quad y^g=\vartheta_2 y_0 , \quad x_0,y_0 \in \GL_n(q_0), \quad \vartheta_i \in \F_{q^2}^*.$$
Recall that we are assuming $n\geq 9$, hence $m\geq 3$.
So,  $x$ has the similarity invariant $t-1$, see \eqref{sim_x}.
It follows that $x_0$ must have the similarity invariant $t-\vartheta_1^{-1}$, whence $\vartheta_1 
\in 
\F_{q_0}$. Similarly, $y$ has the similarity invariant $t^3-1$  by \eqref{sim_y}, hence 
$\vartheta_2^3\in \F_{q_0}$.
Now, the relations 
$\tr(xy)+\tr((yxy)^2)-\frac{\tr((yxy)^3)}{\tr(yxy)}=a$  for $p\neq 3$ and
$\tr(xy)=a$ for $p=3$ (see Lemma \ref{tracce}) give
$a=a_0\vartheta_2$ with $a_0\in \F_{q_0}$, whence $a^3 \in \F_{q_0}$.
Thus $\F_{q^2}=\F_p[a^3]\leq \F_{q_0}$ gives $q_0=q^2$ in contrast with the assumption that
$\F_{q_0}$ is a proper subfield of $\F_{q^2}$.
\end{proof}

\begin{theorem}\label{p<>3}
Suppose  $q\not \in \{2,3,5,7,8,11  \}$ and $n=9,10$ or $n\geq 12$.
Let $a \in \F_{q^2}$ be an element satisfying assumption \eqref{cond2} if $p\neq 3$,  
assumption \eqref{3} if $p=3$, with the additional condition $a^4 -a^3 + a^2 + a-1=0$ when $q=9$.
Then $H=\SU_n(q^2)$.
\end{theorem}

\begin{proof}
The existence of $a$ satisfying \eqref{cond2} or \eqref{3} is guaranteed by 
Propositions \ref{car ne 3} and \ref{policar3}, respectively.
In virtue of these hypothesis,   the group $\langle x,y\rangle$ is absolutely 
irreducible and primitive  by Theorems \ref{Hirr} and \ref{prim}.
Furthermore,  $\langle x,y\rangle$ is also tensor-indecomposable by Lemma \ref{tens}.
Recall that $C=[x,y]^{24}$ is a bireflection: one of the cases (a), (b) or (c) of Theorem  \ref{GS8}  holds for $\langle x,y\rangle$.
If cases (b) or (c) hold, then the projective image of $\langle x,y\rangle$ is contained in a symplectic or in an orthogonal group 
defined over $\F_q$ (see \cite[pp. 409--417]{Ho} and \cite[pp. 186--187]{KL}).
The statement follows now from Proposition \ref{no C5}.
\end{proof}

\section{The case $q=2$, $n\geq 8$ and other cases for $q$ small }\label{eccez}

We start with a general result, used below in a special case. 
For a subset $B$  of the canonical basis $\C$ such that $J_{|B}$ is non-degenerate, we denote by $S_B$ the group acting 
on $\langle 
B\rangle$ as $\SU_{|B|}(q^2)$, with respect to $J_{|B}$, and as the identity on $\langle \C\setminus B\rangle$.

\begin{theorem} \label{Sb}
If $H$ contains $S_{B_\ell}$ with $B_\ell=\{e_{j}\mid \ell\le j\le n\}$ for some $\ell\leq n-4$, then 
$H=\SU_n(q^2)$.
\end{theorem}

\begin{proof} 
Let $k\leq \ell$ be  the smallest positive integer for which $H\ge S_{B_k}$. So, we need to show that $k=1$.
Assume, by contradiction, $k>1$.
Consider first the cases $r=0,1$ or $r=2$ and $k\geq 4$.
Then either $e_{k-1} \in \hat x B_k$ or $e_{k-1} \in y^2 B_k$, where $\hat x=(-1)^n x$.
Setting $g=\hat x$ or $g=y^2$, accordingly, we have 
$$(S_{B_k})^g=S_{\overline B}, \quad \textrm{ with }\quad  \overline B=\left(B_k\setminus
\left\{e_{k}\right\}\right)\cup \left\{e_{k-1}\right\}.$$
Indeed, $(S_{B_k})^g$ fixes $J_{|\overline B}$ and every vector of $\C\setminus \overline B$. Then $(S_{B_k})^g\leq
S_{\overline B}$. Since $|S_{B_k}|=|S_{\overline B}|$, we have the equality.
It follows 
$\left\langle S_{B_k} , (S_{B_k})^g\right\rangle \le S_{B_{k-1}}$. Moreover, as the unique maximal subgroup of
$S_{B_{k-1}}$
which contains $S_{B_k}$ is the stabilizer of  $ e_{k-1}$ we get 
$\left\langle S_{B_k} , (S_{B_k})^g\right\rangle = S_{B_{k-1}}$, 
against the minimality of $k$. 

Finally, when $r=2$ and $k=2,3$, we have 
$$\left\langle S_{B_4}, (S_{B_4})^{\hat x}\right\rangle =S_{{B_4}\cup \left\{e_2\right\}}\equad \left\langle S_{B_3}, 
(S_{B_3})^{\hat x}\right\rangle =S_{{B_3}\cup \left\{e_1\right\}}.$$ 
It follows that
$\left\langle S_{B_4}, (S_{B_4})^{\hat x}, S_{B_3}, (S_{B_3})^{\hat x}\right\rangle =\SU_n(q^2)$. 
\end{proof}
 
In this section we assume that $q=2$ and $n\geq 8$ or $q=3,5,8$ and $n\neq 11$ or $q=7,11$ and $n\neq 8,11$.
Consider the decomposition $V=T_1\oplus T_2$ where 
$$T_1=\langle e_i\mid 1\leq i\leq n-9\rangle \equad T_2=\langle e_i\mid n-8\leq i\leq  n\rangle .$$
When $q \in \{3,5,7,8,11\}$, set $K=\langle \zeta, y\rangle$, with $\zeta=[x,y]^{3}$. By the analysis of $[x,y]$ in Section \ref{gene},
the subspaces $T_1$ and $T_2$ are $K$-invariant for $n=15,16$ and for all $n\geq 18$.
Moreover, $K_{|T_1}$ is contained in $\Sym(11)\times C_3$.
When $q=2$, set $K=\langle \zeta, y\rangle$ with $\zeta=[x,y]^{24}$:
if $n= 10, 14, 15, 16$ or $n\geq 18$ again $T_1$ and $T_2$ are $K$-invariant and 
$\zeta_{|T_1}=\I_{n-9}$.

\begin{proposition}\label{23578}
If $q>2$, let $a$ be an element of $\F_{q^2}$ whose minimal polynomial $m(t)$  over $\F_p$ is as  in {\rm Table 
\ref{tab57}}.
Then $H=\SU_n(q^2)$.
\end{proposition}

\begin{proof}
\textbf{Case 1.} $q=3,5,7,8,11$ and $n=15,16$ or $n\geq 18$.
Set $\tilde \zeta= \zeta_{|T_2}$ and $\tilde y=y_{|T_2}$, and take $\Lambda_q$ as in Table \ref{tab57}.  
Since $\cup_{h \in \Lambda_q}\varpi(h)=\varpi(\SU_{9}(q^2))$, by Lemma \ref{LPS}, we obtain that the projective image of
$\langle \tilde \zeta,\tilde y 
\rangle$ is isomorphic to $\PSU_{9}(q^2)$. 
Note that the order of $\PSU_9(q^2)$ is divisible by a prime $p_1 \geq 13$.
Under our assumptions, there exists $k\in K$ such that $k_{|T_2}$ is non-scalar and has order $p_1$. Clearly $k_{|T_1}$ and $k_{|T_2}$ have coprime orders. 
Hence, $\langle k \rangle = \langle k_{|T_1}\rangle \times \langle k_{|T_2}\rangle$ contains the subgroup generated by $\diag(\I_{n-9}, k_{|T_2})$ of order $p_1$.
It follows that the normal closure $G$ of this subgroup acts as the identity on 
$T_1$ and as $\SU_9(q^2)$ on $T_2$. Our claim now follows from Theorem \ref{Sb}.\\
\noindent \textbf{Case 2.} $q=2$ and $n=9, 10, 14, 15, 16$ or $n\geq 18$.
The subgroup $K'$ acts as the identity on $T_1$ and as $\SU_{9}(2)$ on $T_2$:  the statement follows again from Theorem \ref{Sb}.\\
\noindent \textbf{Case 3.} $q=2$ and $n\in \{8,11,12,13,17\}$ or 
$q=3,5,8$ and $n\in \{8,9,10,12,13,14,$ $17\}$  or 
$q=7,11$ and $n \in \{9,10,12,13,14,17\}$.
Take $\Lambda_{n,q}$ as in Table \ref{tab78}: the statement follows from
Lemma \ref{LPS}, as $\cup_{h \in \Lambda_{n,q}}\varpi(h)=\varpi(\SU_{n}(q^2))$.
\end{proof}

\begin{table}[ht]
$$\begin{array}{|ccc|} \hline
q & m(t) &  \Lambda_q\\ \hline
&&\\[-8pt]
2 & &  \{ (\tilde \zeta^3 \tilde y)^j \tilde y : j =1,6,8,15,33\}\\
3 & t^2-t-1 & \{  (\tilde \zeta \tilde y)^j \tilde y : j=3,5,6, 12, 23  \}\\
5 & t^2-t+2 & \{ (\tilde \zeta \tilde y)^j \tilde y : j= 1,3,8,23,26\} \\
7 &  t^2+t+3  &\{ (\tilde \zeta \tilde y)^j \tilde y : j=1,3,4,7,12,25\} \\
8 & t^6+t^4+t^3+t+1 & \{ (\tilde \zeta \tilde y)^j \tilde y : j=3,4,5,6,9,11  \}\\
11 &  t^2+7t+2 & \{ (\tilde \zeta \tilde y)^j \tilde y : j= 5,6,7,8,18\} \\\hline
 \end{array}$$ 
\caption{Sets $\Lambda_q$ for $q=2,3,5,7,8,11$.} 
\label{tab57}
\end{table}

\begin{table}[ht]
$$\begin{array}{|ccc|} \hline
 n & \Gamma_{n,2}&  \Gamma_{n,3}  \\ \hline
 8 & \{1,2,6,8 \}  & \{1,4,5,8\}\\
9 &  & \{1,2,4,8,15\} \\ 
10 &  & \{1, 2,11,13,17,21 \} \\ 
11 & \{1, 2, 6, 7, 8,  30\} & \\
 12 &\{4,6,10,20,30,46\} &  \{2,4,5,6,13,29 \} \\ 
13 &  \{1,6,7,8,21,39,44\} & \{ 2,3,5,6,8,11,13 \} \\ 
14 & & \{1,2,4,5,23, 30,32 \} \\ 
17& \{2,4,6,7,13,14,17,19,44,46 \} &  \{4,10,12,15,16,18,19,24,  33, 40 \} \\\hline
 n &  \Gamma_{n,5} & \Gamma_{n,8}\\ \hline
 8 & \{4,10,11,13\} & \{1,2,4,7,34\}\\
9 & \{ 1,2,4,8,13 \} & \{2,5, 8, 10,11\}  \\ 
 10 & \{ 2,7,10,15,16 \} &  \{1,3,5,7,13,14,17\}\\ 
 12 & \{4,7,10,13,25,27 \} & \{1,2, 4, 9,14,15, 43 \}  \\ 
13 & \{ 3,6, 13,14,15,20,23 \} &\{2,4,6,8,11,18,26,40 \} \\ 
 14 & \{2,4,5,7,14,19,23  \} & \{1,4,5,7,16,19,34\} \\ 
 17&\{5,9,13,25,27,28,32, 33,41  \}&  \{2,5,7,9,10,12,13,18, 21,41\}    \\\hline
   n &  \Gamma_{n,7} & \Gamma_{n,11}\\ \hline
9 &   \{1,2,3,4,6,11 \} & \{1,3,13,16,24\}\\ 
 10  & \{3,4,5,18, 35 \}  & \{2,3,7,12,13\}\\ 
 12 & \{1,2,7,22,26,33\}  & \{1,6,9,16,19,69\}\\ 
13 &   \{3,8,9,15,17,21,25 \} & \{2,7,8,11,21,22,25,41\}\\ 
 14 & \{1,4,6,9,10, 26,29,  35 \} & \{3,4,6,15,20,30,35,37\}\\ 
 17&\{4,6,9,11,13,14,15,18,19,31  \}    & \{1,4, 5,6,13,14,15,21,23\}\\\hline
\end{array}$$ 
\caption{Sets $\Lambda_{n,q}=\{[x,y] (xy)^j : j \in \Gamma_{n,q} \}$ for $q=2,3,5,7,8,11$.} 
\label{tab78}
\end{table}

\section{The case $n=8$ and $q\neq 2,3,5,8$}\label{8}

In this section we suppose that $n=8$, $q\ne 2,3,5,8$ and that $a$ satisfies
conditions \eqref{3} when $p=3$ and the  following conditions if $p\neq 3$:
\begin{itemize}
\item[\rm{(1)}] $\F_p[a^3]=\F_{q^2}$ and $a^{q+1}=1$;
\item[\rm{(2)}] $4a^4-11a^3+24 a^2-11 a+4\neq 0$;
\item[\rm{(3)}] $a=\iota$, a primitive fourth root of unity,  if $q=p\equiv 3 \pmod 4$.
\end{itemize}

When $p\ne 3$, we have $\gamma=\frac{(a^3+1)^2}{a^3}$, with 
$a^3+1\ne 0$ by (1), and so condition \eqref{1} holds.
The value of $a$ in (3) satisfies
(1) and  (2). By  Lemma \ref{Nbeta} there exists   $a\in \F_{q^2}$ satisfying the above conditions
for all the values of $q$ under consideration except for $q=4$ and $q=9$,  in which cases we can take  $a$ 
respectively with minimal polynomial $t^4+t^3+t^2+t+1$ and $t^4-t^3+t^2+t-1$ over $\F_p$.
The characteristic polynomial of $[x,y]$ is $\chi(t)=(t^2 + t +1)\chi_0(t)$, with
\begin{equation*}
\chi_0(t)=t^6-\frac{\gamma}{9} t^5 -\frac{\gamma}{9} t^4+ \frac{2\gamma+3 b+3b^q}{9} t^3  - \frac{\gamma}{9} 
t^2-\frac{\gamma}{9} t+1.
\end{equation*}
Under assumption (2) above, the factors $t^2 + t +1$ and $\chi_0(t)$
are coprime.
It follows that $\omega^{\pm 1}$ are simple roots of $\chi(t)$ and
the eigenspaces of $[x,y]$ and
$[x,y]^T$ relative to $\omega$ are  
respectively generated by: 
\begin{equation}\label{somega p ne3}s_{\omega} =\left( 1, \omega, -1, -\omega^2, -\omega, \omega^2, 
\frac{3 a^2 (2a-1)\omega }{(\omega - 1)  (a^2-a+1)^2 },
\frac{ 3 a (2a-1)}{(2\omega+1)  (a^2-a+1)^2 }\right),
\end{equation}
and 
\begin{equation}\label{sbaromega p ne3}\overline{s}_{\omega} =\left( 1, \omega^2, -1, -\omega , -\omega^2 ,\omega, 
\frac{(a+1)^2 (2-a)}{3a^2}, \frac{(a+1)^2 (2-a)\omega}{3a}\right).
\end{equation}

When $p=3$, the characteristic polynomial of $B= \left([x,y]^2y\right)^3 y$ is $(t-1)\chi_1(t)$, where
$$\chi_1(1)=-\frac{a^{13}(a-1)}{(a+1)^7}.$$ 
It follows that $1$ is a simple eigenvalue of $B$ and $B^T$. The corresponding eigenspaces are generated
respectively by 
\begin{equation}\label{somega p eq 3}s_{1} =\left( a, a^q,1,1,1, 0,0,0 \right),
\equad \overline{s}_{1} =\left(1,-( a+1), a^{-q}, a^{-q}, a^{-q},0,0,0 \right). 
\end{equation}

\begin{lemma}\label{irr8} 
The subgroup $H=\langle x, y\rangle$ is absolutely irreducible.
\end{lemma}

\begin{proof}
Let $U$ be an $H$-invariant subspace of $V=\F^8$ and $\overline U$ be the corresponding $H^T$-invariant subspace.
Assume first $p\ne 3$.\\
\noindent \textbf{Case 1.} The restriction $[x,y]_{|U}$ has the eigenvalue $\omega$.
Then $s_{\omega}$ as in \eqref{somega p ne3} belongs to $U$ and $xs_{\omega}=- s_{\omega^{-1}}\in U$.
Consider the matrix $M$ whose columns are the vectors
\begin{equation}\label{vett}
s_{\omega},\; s_{\omega^{-1}},\; ys_{\omega}, \; ys_{\omega^{-1}},\; y^2s_{\omega},\; 
y^2s_{\omega^{-1}},\; xy^2s_{\omega},\; xy^2s_{\omega^{-1}}.
\end{equation}
Then $\det(M)=\frac{ 81 a (2 a-1) (a+1)^4}{(a^2-a+1)^2}\neq 0$ and so the vectors in \eqref{vett}
are linearly independent. We conclude that $U=\F^8$. \\
\noindent\textbf{Case 2.} If Case 1 does not occur, then $[x,y]^T_{|\overline{U}}$ has the eigenvalue $\omega$. 
It follows that $\overline{s}_{\omega} \in \overline{U}$ and $x^T \overline{s}_{\omega}=-\overline{s}_{\omega^{-1}} \in 
\overline{U}$.
Consider the matrix $N$  whose columns are the vectors
\begin{equation}\label{vett2}
\overline{s}_{\omega},\; \overline{s}_{\omega^{-1}},\; y^T\overline{s}_{\omega},\; 
y^T\overline{s}_{\omega^{-1}}, \;(y^T)^2\overline{s}_{\omega},\;
(y^T)^2\overline{s}_{\omega^{-1}},\;  x^Ty^T\overline{s}_{\omega},\;
x^Ty^T\overline{s}_{\omega^{-1}}.
\end{equation}
Then $\det(N)=\frac{-3\gamma (a+1)^6 (a-2)}{a^5}\neq 0$  and hence  the vectors in \eqref{vett2} are linearly
independent. We conclude that $\overline{U}=\F^8$ and hence $U=\{0\}$.

Now suppose $p=3$.\\
\noindent \textbf{Case 1.} The restriction $B_{|U}$ has the eigenvalue $1$.
Then $s_1 \in U$.
Consider the matrix $M$ whose columns are the vectors
\begin{equation}\label{vett83}
s_1, \;xs_1,\; yxs_1,\; y^2xs_1,\; xyx s_1, \;xy^2x s_1,\; (yx)^2s_1,\; [y^2,x]s_1.
\end{equation}
Then $\det(M)=\frac{a^{15}(a-1)}{(a+1)^9}\neq 0$.
So, the vectors in \eqref{vett83} are linearly
independent and we conclude that  $U=\F^8$. \\
\textbf{Case 2.} If Case 1 does not occur, then $(B^T)_{|\overline{U}}$ has the eigenvalue $1$. 
It follows that $\overline{s}_{1} \in \overline{U}$. 
Consider the matrix $N$ whose columns are the vectors 
\begin{equation}\label{vett84}
\begin{array}{c}
\overline{s}_1, \;x^T\overline{s}_1,\; y^Tx^T\overline{s}_1,\; (y^T)^2x^T\overline{s}_1,\; x^Ty^Tx^T \overline{s}_1,
\;x^T(y^T)^2x^T \overline{s}_1,\\ (y^Tx^T)^2\overline{s}_1,\; [(y^T)^2,x^T]\overline{s}_1.
\end{array}
\end{equation}
Then $\det(N)=-a^7(a^2-1)\neq 0$, and so  the vectors in \eqref{vett84} are linearly
independent.
We conclude that $\overline{U}=\F^8$ and hence $U=\{0\}$.
\end{proof}

When $p\ne 3$, the characteristic polynomials of $xy$ and $xy^{-1}$ are:
\begin{align}
\label{charxy}
\psi_{1}(t)= t^8 - \frac{a^3+1}{3a^2}  t^7 - \frac{a^3-2}{3a} t^6 + \frac{2a^3-1}{3a^2} t^2 
- \frac{a^3+1}{3a} t+ 1, \\
\nonumber
\psi_{-1}(t)= t^8 - \frac{a^3+1}{3a}  t^7 +    \frac{2a^3-1}{3a^2} t^6- \frac{a^3-2}{3a} t^2 
- \frac{a^3+1}{3a^2} t+ 1.
\end{align}
When $p=3$, the characteristic polynomials of $xy$ and $xy^{-1}$ are:
\begin{equation}\label{charxy3} 
\begin{array}{rcl}
\psi_1(t) &=&  t^8 -at^7 -t^6 -t^2 +\frac{a}{a+1} t + 1,\\
\psi_{-1}(t) & =&  t^8 +\frac{a}{a+1}t^7 -t^6 -t^2 -at + 1.
\end{array}
 \end{equation}
 
\begin{lemma}\label{prim8}
The subgroup $H$ is primitive.
\end{lemma}

\begin{proof}
Suppose, by contradiction, that  $V=V_1\oplus \ldots \oplus V_\ell$ is a decomposition permuted by $H$.
By the irreducibility (Lemma \ref{irr8}), the permutation action is transitive. In particular $\dim 
V_i=\frac{8}{\ell}$, $1\leq i\leq \ell$. 
Since $\ell \not\equiv 0\pmod 3$, in all cases we may set $yV_1=V_1$.\\
\noindent \textbf{Case 1.} $\ell=8$. Set $V_i=\langle v_i\rangle = i$.
We may suppose 
$xv_1=v_3$, $yv_3=v_4$, $yv_4=v_5$. Moreover, substituting $y$ with $y^2$ if necessary,
$xv_5=v_6$, $yv_6=v_7$, $yv_7=v_8$. It follows $y v_2=v_2$. 
We consider the various possibilities for the restriction $x_{|\{2,4,7,8\}}$ of $x$ to $\{2,4,7,8\}$ 
and the corresponding characteristic polynomial $\psi(t)$ of $xy$.
$$\begin{array}{cc|cc}
x_{|\{2,4,7,8\}}  & xy  & x_{|\{2,4,7,8\}} & xy  \\\hline
(2,4) & (1,3,2,4,6,7,8,5) & 
(2,4)(7,8) & (1,3,2,4,6,8,5)\\
(2, 7)& (1,3,4,6,2,7,8,5) &
(2,7)(4,8)  & (1,3,8,5)(2,7,4,6)  \\
(2 ,8)& (1,3,4,6,7,2,8,5) &
(2,8)(4,7) & (1,3,7,2,8,5)(4,6)
\end{array}$$
In the three cases on the left $\psi(t)$ has shape $t^8+\lambda$. In the cases on the right $\psi(t)$ has 
respective shapes $t^8+\lambda t^7+\mu t+\nu$, $t^8+\lambda t^4+\mu$ and $t^8+\lambda t^6+\mu t^2+\nu$. 
In all cases the term of degree $1$ or $2$ is  missing. Comparison with $\psi_{\pm 1}(t)$ gives the contradiction $a^3\in \F_p$
(or $1=0$ when $p=3$ and $\psi(t)=t^8+\lambda t^7+\mu t+\nu$). \\
\noindent \textbf{Case 2.} $\ell=4$. Set $V_1=\langle v_1,v_2\rangle$, $V_2=\langle v_3,v_6\rangle$, $V_3=\langle v_4,v_7\rangle$ 
and $V_4=\langle v_5,v_8\rangle$.
We may assume $xV_1=V_2$, $yV_2=V_3$ and $yV_3=V_4$.
Hence, considering the similarity invariants of $y$, we may assume $yv_i= v_i$, $i=1,2$, and 
$xv_1=v_3$, $xv_2=v_6$, $yv_3=v_4$, $yv_4=v_5$, $yv_6=v_7$ and $yv_7=v_8$.
Moreover, we have either $(i)$ $xV_3=V_3$, $xV_4=V_4$, or $(ii)$ $xV_3=V_4$.
In  case $(i)$, setting $xv_4=\lambda v_4+\mu v_7$, $xv_7=\nu v_4+ \zeta v_7$, 
$xv_5=\varepsilon v_5 + \xi v_8$, $xv_8= \tau v_5 + \varsigma v_8$,
we get that $\psi(t)$ has shape $t^8+\alpha t^4+\beta$.
In case $(ii)$, set $xv_4=\lambda v_5+\mu v_8$, $xv_7=\nu v_5+ \tau v_8$, 
$x v_5= \varepsilon v_4+ \xi v_7$, $xv_8=  \zeta v_4 + \varsigma v_7$.
If $\varepsilon=0$, then $\tau=0$, $\zeta=\mu^{-1}$, $\xi=\nu^{-1}$
and $\lambda=-\varsigma \mu \nu$, while if $\varepsilon\neq 0$, then $\mu=-\xi\tau\varepsilon^{-1}$, 
$\nu= -\tau\zeta\varepsilon^{-1}$ and $\lambda=\frac{\xi\tau\zeta + \varepsilon}{\varepsilon^2}$.
We get that the coefficients of $t^5$ and $t$ in $\psi(t)$ coincide.
In all cases, comparison with $\psi_{\pm 1}(t)$ gives the absurd $a^3\in \F_p$. \\
\noindent \textbf{Case 3.} $\ell=2$. Clearly $yV_2=V_2$, $xV_1=V_2$ whence $[x,y]V_1=V_1$, $[x,y]V_2=V_2$.
When $p\ne 3$ we may suppose $s_{\omega}\in V_1$. 
It follows that the subspace generated by 
$$s_{\omega},\; ys_{\omega},\; y^2s_{\omega}, \; xy^2xs_{\omega},\;  xyxy^2s_{\omega}$$
is contained in $V_1$.
The matrix consisting of coordinates $1, 2, 3, 5, 8$ of these five vectors 
has determinant $-\omega^2  \frac{(a+1)^4(a+\omega^2)}{a}\neq 0$, an absurd.

When $p=3$ it is easy to see that $B$ fixes both $V_1$ and $V_2$. So we may assume that $s_1\in V_1$.
Then the five  linearly  independent vectors 
$$s_1, \; xyxs_1, \;   xy^2xs_1, \;  (yx)^2 s_1, \; [y^2,x]s_1 $$
of \eqref{vett83} lie in $V_1$, an absurd.
 \end{proof}

\begin{lemma}\label{tensor8}
The subgroup $H$ is  tensor-indecomposable.
\end{lemma}

\begin{proof}
Suppose, by contradiction,  that $H$  is conjugate to a subgroup of $\GL_2(\F)\otimes \GL_4(\F)$
and set $x=\tilde x_2 \otimes \tilde x_4$ and $y=\tilde y_2\otimes \tilde y_4$.
We must have  $\tilde x_2$ and $\tilde y_2$ non scalar, otherwise the group $\langle \tilde x_2,\tilde y_2\rangle$ 
would be reducible, in contrast with the 
irreducibility of $H$ (Lemma \ref{irr8}). Since the fixed points space of $y$ has dimension $4$, 
we may assume that there is a $2$-dimensional subspace $W$ of $\F^4$ on which $\tilde y_4$ acts as the 
identity. By the irreducibility of $H$ we must have $xW\cap W=\{0\}$.
Let $w_1,w_2$ be a basis of $W$. Then $\B_4=\{w_1,w_2,\tilde x_4 w_1, \tilde x_4 w_2 \}$ is a basis of $\F^4$.
Call $\B_2=\{v_1,v_2\}$ a basis of $\F^2$ on which $\tilde y_2$ acts in Jordan form. 
With respect to the basis $\B_2\otimes \B_4$ we have:
$$x=\begin{pmatrix}
    \alpha_1 & \alpha_2 \\ \alpha_3 & -\alpha_1  
    \end{pmatrix}\otimes \begin{pmatrix} 0 & \I_2 \\ \I_2 & 0\end{pmatrix}
   \equad  y=\begin{pmatrix} 1  & \lambda \\ 0 & \omega^{-1} \end{pmatrix}\otimes
    \begin{pmatrix} 1 & 0 & \beta_1 & \beta_2 \\ 0 & 1 & \beta_3 & \beta_4 \\
    0 & 0 & \omega & \nu\\ 0 & 0 & 0 & \omega
    \end{pmatrix}
$$
with $\lambda=\nu=0$ if $p\ne 3$,  $\omega=\lambda=1$ and $\nu\in\left\{0,1\right\}$ if $p=3$.
By the irreducibility of $H$ we get $\alpha_3\ne 0$ and, if $p\ne 3$, we also have $\alpha_2\ne 0$. 
Conjugating by an element of the centralizer of $\tilde y_2$ we may suppose 
$\alpha_2=1$, $\alpha_3=1-\alpha_1^2$.

When $p\ne 3$, computing  the characteristic polynomial of $(\tilde x_2 
\otimes \tilde x_4)(\tilde y_2\otimes \tilde y_4)$ we obtain that the coefficients of the terms of degree $2$ and $6$ 
are equal. Then, from  \eqref{charxy} it follows $a=\pm 1$, an absurd.

When $p= 3$ and $\nu =1$ we must have $\beta_1=\beta_3=0$ in order that the fixed points space of $\tilde y_2\otimes \tilde y_4$
has dimension $4$.  
For both values of  $\nu =0,1$  the characteristic polynomial of $(\tilde x_2 
\otimes \tilde x_4)(\tilde y_2\otimes \tilde y_4)$ has  the coefficients of the terms of degree $1$ and $7$ 
which are equal. Comparison with \eqref{charxy3} gives the contradiction $a\in \left\{0,-1\right\}$.
\end{proof}
 
\begin{lemma}\label{no C58}
The subgroup $H$ is not contained in any maximal subgroup $M$ of class $\Cl_5$ of $\SU_8(q^2)$.
\end{lemma}

\begin{proof}
Suppose, by contradiction, that there exists $g \in \U_8(q^2)$ such that 
$$\langle x, y \rangle^g\leq  \GL_8(q_0)\left( \F_{q^2}^* \I_8\right) \leq \SU_8(q^2),$$
where $\F_{q_0}$ is a proper subfield of $\F_{q^2}$. Set
$$x^g = \vartheta_1 x_0,\quad y^g=\vartheta_2 y_0 , \quad x_0,y_0 \in \GL_8(q_0), \quad \vartheta_i \in \F_{q^2}^*.$$
Since  $y$ has the similarity invariant $t-1$, we obtain that $\vartheta_2\in \F_{q_0}$.
Now, for $p\ne 3$, from \eqref{charxy} we get $\tr (xy)\ne 0$  and, by a  direct calculation, $\frac{\tr(yxy)}{\tr(xy)}=a$. 
For $p=3$, from \eqref{charxy3} we get $-\frac{\tr(xy)}{\tr(xy^2)}=a+1$.
In both cases, $a\in \F_{q_0}$.
Thus $\F_{q^2}=\F_p[a^3]\leq \F_{q_0}$ gives $q_0=q^2$, against the assumption that
$\F_{q_0}$ is a proper subfield of $\F_{q^2}$.
\end{proof}

\begin{lemma}\label{8C6}
The subgroup $H$ is not contained in any maximal subgroup $M$ of class $\Cl_6$.
\end{lemma}

\begin{proof}
Suppose $H\le M$. By the classification of \cite[Tables 8.46 and 8.47]{Ho} we may assume
$q=p\equiv 3 \pmod 4$: in particular $a=\iota$ by the assumptions at the beginning of this section. Moreover  
$\frac{M}{E}\cong \Sp_6(2)$, where $E=Z_d\circ 2^{1+6}$ is an absolutely irreducible $2$-group in which all squares are 
scalar. Here, $Z_d$ is a cyclic group of order $d=(p+1,8)$.
Call $g=E[x,y]$ the image of $[x,y]$ in $\Sp_6(2)$ and $m$ the order $g$.
Recall that $[x,y]$ has the eigenvalues $\omega^{\pm 1}$.
So, from $[x,y]^m\in E$ it follows that $3$ divides $m$ and, moreover, that $[x,y]^{2m}=\I_8$.
Inspection of the orders of the elements in $\Sp_6(2)$ gives $m \in \{3,6,9,12,15\}$.
If $m=3,6,12$, then $[x,y]^{24}=\I_8$: consideration of the entry of position $(3,1)$ excludes this possibility.
If $m=9$, then $[x,y]^{18}=\I_8$, which can be excluded considering the entry of position $(3,2)$.
Finally, the case $m=15$ can be excluded looking at the entries of $[x,y]^{30}=\I_8$ 
of position $(1,4)$ and $(1,5)$.
\end{proof}

\begin{lemma}\label{8S}
The subgroup $H$ is not contained in any maximal subgroup $M$ of class $\mathcal{S}$.
\end{lemma}

\begin{proof}
Suppose $H\leq M$. As in the previous lemma, by \cite[Table 8.47]{Ho}, we may assume $11\leq q=p\equiv 3 \pmod 4$,
$a=\iota$.  Moreover
$N_0.\PSL_3(4)\leq M \leq N_0.\PSL_3(4).2_3$, where $N_0$ is a $2$-group such that
$N_0^4\leq \langle -\I \rangle$. We recall that 
the commutator $[x,y]$ must belong to the derived subgroup, hence $[x,y]\in N_0.\PSL_3(4)$.
The elements of $\PSL_3(4)$ have orders  $1, 2, 3, 4, 5, 7$.
Let $s$ be the order of $N_0[x,y]$ in $\PSL_3(4)$.
Then, the order of $[x,y]$ divides $2^\alpha s$. Since $3$ divides the order of $[x,y]$, the only possibility is that 
$N_0[x,y]$ has order $3$, namely that $[x,y]^{12} = \pm \I_8$.
Consideration of the entry of position $(1,3)$ leads to an absurd.
 \end{proof}
 
 We can now conclude, recalling all our assumptions for the reader's convenience.

\begin{proposition}\label{p8no3}
Suppose $q\ne 2,3,5,8$. If $p\ne 3$,  let $a \in \F_{q^2}$ be such that
\begin{itemize}
\item[(1)] $\F_{q^2}=\F_p[a^3]$ and  $a^{q+1}=1$;
\item[(2)] $4a^4-11a^3+24 a^2-11 a+4\neq 0$.
\end{itemize}
Moreover if $q=p\equiv 3 \pmod 4$, let $a$ be a primitive fourth root of unity.

If $p=3$, let $a \in \F_{q^2}$ be such that 
$$
\F_{q^2}=\F_3[a] \equad a^{q}+a^{q-1}+1=0.
$$
Then the group $H=\langle x,y\rangle$ coincides with $\SU_{8}(q^2)$.
\end{proposition}

\begin{proof}
The subgroup $H$ is absolutely irreducible by Lemma \ref{irr8}.
It follows from Lemmas \ref{prim8} and \ref{tensor8} that $H$ is primitive and tensor-indecomposa\-ble.
Suppose that $H$ is contained in a maximal
subgroup $M$ of 
$\SU_{8}(q^2)$. According to the classification of \cite[Tables 8.46 and 8.47]{Ho}, $M$ is then a subgroup belonging to
$\Cl_5\cup \Cl_6\cup \mathcal{S}$. 
Classes $\Cl_6$ and $\mathcal{S}$, which must be considered only when $q=p\equiv 3 \pmod 4$, are ruled out by 
Lemmas \ref{8C6} and \ref{8S}. Finally, class $\Cl_5$ is excluded by Proposition \ref{no C58}.
We conclude that $H=\SU_{8}(q^2)$.
\end{proof}

\section{The case $n=11$ and $q>2$}\label{11}

Our generators work also for $n=11$ but, for a shorter proof,
we  use different ones. More precisely we extend to $n=11$ the generators used in \cite{SU7} for $n=7$.
So, let  $\C=\{e_1,\ldots,e_{11}\}$ be the 
canonical basis of
$V=\F^{11}$.
Let $a \in \F_{q^2}$ be such that $\F_{q^2}=\F_p[a]$ and 
define $H=\langle x,y \rangle$, where $x$ and $y$  are matrices of respective order $2$ and $3$, acting 
on $\C$ as follows:
\begin{itemize}
\item $x e_{2j+1}= e_{2j+2}$ for any $j=0,\ldots,4$;
\item $xe_{11}=a(e_1+e_2)+e_5+e_6-(e_9+e_{10}+e_{11})$;
\item $ye_1=e_1$ and $ye_{11}=(a+a^q+1)e_1-(e_{10}+e_{11})$;
\item $y e_{2j}=e_{2j+1}$ for any $j=1,\ldots,5$;
\item $y e_{2j+1}=-(e_1+e_{2j}+e_{2j+1})$  and $y e_{2j+5}=e_1-(e_{2j+4}+e_{2j+5})$ for any $j=1,2$.
\end{itemize}
The similarity invariants of $x$ and $y$ are, respectively,
$$d_1(t)=t+1, \quad d_2(t)=d_3(t)=d_4(t)=d_5(t)=d_6(t)=t^2-1
$$
and
\begin{equation}\label{sim11y}
d_1(t)=d_2(t)=d_3(t)=d_4(t)=t^2 + t + 1, \quad d_5(t)=t^3-1.
\end{equation}
We then obtain that $x,y \in \SL_{11}(q^2)$.
The characteristic polynomial of $z=xy$ is
$$\chi_z(t)=t^{11} - t^9 + 2 t^7 - (a + 1) t^6 + (a^q + 1) t^5 - 2 t^4 + t^2 - 1.$$

\begin{lemma}\label{uni11}
If $H$ is absolutely irreducible, then $H\leq \SU_{11}(q^2)$.
\end{lemma}

\begin{proof}
Let $C(g)$ be the centralizer of $g$ in $\Mat_{11}(q^2)$.
Applying Frobenius formula \cite[Theorem 3.16]{Jac}, we obtain that $\dim C(x)=61$ and
$\dim C(y)=51$. By the absolute irreducibility of $H$ and applying Scott's formula (see \cite{SU4}), 
we obtain that $\dim C(z)=11 $ and, in particular, that $z$ has a unique similarity invariant.
From \cite[Theorem 3.1]{SU4} it follows that $H\leq \SU_{11}(q^2)$.
\end{proof}

We now consider the absolute irreducibility of $H$. Define
$$\begin{array}{rcl}
f_1(a) & = & a^{2q}- a^{q+1} - 3a^q +a^2   - 3 a  + 9, \\
f_2(a) & =&   a^{5q}+ 5 a^{4q+1} + 16 a^{4q} + 10 a^{3q+2} + 64 a^{3q+1}  + 92 a^{3q}
+ 10 a^{2q+3} +\\
&& 128 a^{2q+2} + 436 a^{2q+1}+ 424 a^{2q} + 5 a^{q+4} + 64a^{q+3}+ 436 a^{q+2} +\\
&& 1168 a^{q+1}+ 1008 a^q  + a^5 + 16 a^4   + 92 a^3    + 424 a^2    + 1008a      +        864.
\end{array}$$
        
 \begin{lemma}\label{irr11}
Let $a\in \F_{q^2}$ be such that $f_1(a)f_2(a)\neq 0$. Then $H$ is absolutely irreducible.
 \end{lemma}

  \begin{proof}
 Assume that $f_1(a)f_2(a)\neq 0$ and let $U\neq V$ be an $H$-invariant subspace of $V$.
 A direct calculation shows that, for all $u \in V$, the element $u+yu+y^2u$ always belongs to the subspace $\langle 
e_1\rangle$. On the other hand, we have
$x e_{2i+1}=e_{2i+2}$ and $ye_{2i+2}=e_{2i+3}$ for all $i=0,\ldots,4$.
It follows that, if $v=u+yu+y^2u\neq 0$ for some $u \in U$, then the $H$-submodule generated by $v$ is the whole space 
$V$, in contradiction with the assumption $U\neq V$.
Hence, every element $u \in U$ satisfies the following condition:
\begin{equation}\label{3.6}
u + yu + y^2 u=0.
\end{equation}
We will show that this condition implies $U=\{0\}$.\\
\textbf{Case 1.} Suppose that $u+xu=0$ for all $u \in U$. Then all vectors in $U$ have shape
$$(\lambda_1,a(\lambda_5+\lambda_6)-\lambda_1, \lambda_3,-\lambda_3,\lambda_5,\lambda_6,\lambda_7,-\lambda_7,
\lambda_9, -(\lambda_5+\lambda_6+\lambda_9), -(\lambda_5+\lambda_6)).$$
Fix a nonzero $u \in U$. From $yu+x(yu)=0$ and $y^2u+x(y^2u)=0$ we get $\lambda_5=-(\lambda_1+(a+1)\lambda_3)$,
$\lambda_6=\lambda_1+a \lambda_3$, $\lambda_7=-\lambda_3$, $\lambda_9=\lambda_1+(a+1)\lambda_3$ and
$(\lambda_1,\lambda_3)A=(0,0)$, where
$$A=\begin{pmatrix}
  a-3 & a^q-3\\
  a^2-2a-a^q & a^{q+1}-2a-3
  \end{pmatrix}.$$
Since $\det(A)=f_1(a)\neq 0$, we obtain $\lambda_1=\lambda_3=0$ and hence $u=0$. This means that $U=\{0\}$, as 
desired.\\
\textbf{Case 2.} Suppose that there exists $v \in U$ such that $v+xv\neq 0$. Then, the vector $w=v+xv$ has shape
$(\lambda_1,\lambda_1,\lambda_3,\lambda_3, \lambda_5,\lambda_5, \lambda_7,\lambda_7, \lambda_9,\lambda_9,0 )$.
Equation \eqref{3.6} applied to $w$ gives the condition $2\lambda_3=2 (\lambda_1+\lambda_7)+(a+a^q+2)\lambda_9$.
Assume first that $p=2$. Since $f_2(a)=(a+a^q)^5$ is nonzero by hypothesis, it follows that $\lambda_9=0$.
Application of \eqref{3.6} to the vectors $(xy)^iw \in U$ for $i=1,2,3,4$ gives $\lambda_{9-2i}=0$.
So, $u=0$ and hence $U=\{0\}$. Now, assume $p\neq 2$ and set $\lambda_3=\lambda_1+\lambda_7+ \frac{a+a^q+2}{2} 
\lambda_9$. 
Application of \eqref{3.6} to the vector $ xy w \in U$ gives $\lambda_1=\lambda_5 +\frac{a+a^q}{2}\lambda_7-
\frac{a+3a^q+6}{2}\lambda_9$. Finally, application of the same equation to the vectors $ (xy)^i w \in U$ for $i=2,3,4$ 
gives
$(\lambda_5, \lambda_7,\lambda_9)B=(0,0,0)$,
where $B$ is the following matrix
\begin{small}
$$\begin{pmatrix}
 a+a^q+4 &  4(a^q+2) &  4 (a^q+3) \\ 
 -2(a^q+2) &  -(a+a^q+2)^2-8(a^q+2) &  -2(a^{2q}+ a^{q+1}+7a^q +  a+12) \\
 2(a^q+3) &  2 (a^{2q}+7 a^q+a^{q+1}+a+12) &  -(a+a^q)^2 +4(a^{2q}-2a+2a^q+3) 
  \end{pmatrix}.$$

  \end{small}
\noindent Since $\det(B)=f_2(a)\neq 0$, we obtain $\lambda_5=\lambda_7=\lambda_9=0$ and so $u=0$.
It follows that also in this case $U=\{0\}$, proving that $H$ is absolutely irreducible.
  \end{proof}

\begin{lemma}\label{mon11}
If $H$ is absolutely irreducible, then $H$ is not monomial.
\end{lemma}

\begin{proof}
Let $\{v_1,v_2,\ldots,v_{11}\}$ be a basis on which $H$ acts as a monomial, transitive group.
We may then assume  $yv_1=v_2$, $yv_2=v_3$, $yv_3=\lambda_1 v_1$, $yv_4=v_5$, $yv_5=v_6$, $yv_6=\lambda_4 v_4$,
$yv_7=v_8$, $yv_8=v_{9}$ and $yv_9=\lambda_7 v_7$.
As $y^3=\I_{11}$, we must have $\lambda_1=\lambda_4=\lambda_7=1$.
This implies that $y$ has the similarity invariant $t^3-1$ with multiplicity $3$, a contradiction with \eqref{sim11y}.
\end{proof}

\begin{lemma}\label{sub11}
Assume that $H$ is absolutely irreducible. 
Then, $H$ is not contained in any maximal subgroup of
$\SU_{11}(q^2)$ of class $\Cl_5$.
\end{lemma}

\begin{proof}
Suppose, by contradiction, that there exists $g \in \U_{11}(q^2)$ such that 
$$\langle x,y\rangle^g\leq  \GL_{11}(q_0)\left( \F_{q^2}^* \I_{11}\right) \leq \SU_{11}(q^2),$$
where $\F_{q_0}$ is a proper subfield of $\F_{q^2}$. Set
$x^g = \vartheta_1 x_0$, $y^g=\vartheta_2 y_0$, where $x_0,y_0 \in \GL_{11}(q_0)$ and $\vartheta_i \in \F_{q^2}^*$.
Recall that  $x$ has the similarity invariant $t+1$ and $y$ has the similarity invariant $t^2+t+1$. It follows that $x_0$ and $y_0$ must have, respectively, the similarity invariants $t+\vartheta_1^{-1}$ and $t^2+\vartheta_2^{-1} t
+\vartheta_2^{-1}$: we conclude that $\vartheta_1,\vartheta_2 \in \F_{q_0}$. 
Now, the relation $\tr((xy)^9)=-9a$  gives $a\in \F_{q_0}$, unless $p=3$.
In this case, we use $\tr((xy)^5)=-(a+1)$, which gives again $a \in \F_{q_0}$.
Thus, $\F_{q^2}=\F_p[a]\leq \F_{q_0}$ returns $q_0=q^2$ in contrast with the assumption that
$\F_{q_0}$ is a proper subfield of $\F_{q^2}$.
\end{proof}

\begin{lemma}\label{S11}
Assume that $H$ is absolutely irreducible. 
Then, $H$ is not contained in any maximal subgroup $M$ of
$\SU_{11}(q^2)$ of class $\mathcal{S}$.
\end{lemma}

\begin{proof}
Suppose that $H\leq M$. By \cite[Tables 8.73]{Ho} we have $q=p\geq 5$
and $M = Z\times K$,  where $Z$ has order $\gcd(q+1,11)$ and 
$K\in \{\PSL_2(23), \PSU_5(4)\}$.
It easily follows that $x,y\in K$. Note that $\tr(x) = -1$ and $\tr(y) = -4$. 
Considering the irreducible characters $\chi$ of $\PSL_2(23)$  and $\psi$ of $\PSU_5(4)$ of degree $11$, we see that if $g$ has order 3 then $\chi(g) = -1$, in contrast with $\tr(y) = -4$ (since $p\neq 3$); if $g$ has order $2$ then $\psi(g) 
\in \{-5,3\}$, in contrast with $\tr(x) = -1$ (since $p\neq 2$).
\end{proof}

\begin{proposition}\label{pr11}
Take $x,y$ as before and let $H=\langle x,y\rangle$.
If $a \in \F_{q^2}$  is such that 
\begin{equation}\label{11g}
\F_{q^2}=\F_p[a] \equad  f_1(a)f_2(a)\neq 0,
\end{equation}
 then $H=\SU_{11}(q^2)$.
Moreover, if $q> 2$, then there exists $a \in \F_{q^2}$ satisfying  \eqref{11g}.
\end{proposition}

\begin{proof}
Since $f_1(a)f_2(a)\neq 0$, the group $H$ is absolutely irreducible, as shown in Lemma \ref{irr11} and hence, by Lemma
\ref{uni11},  $H$ is subgroup of $\SU_{11}(q^2)$. Suppose that $H$ is contained in a maximal subgroup $M$ of 
$\SU_{11}(q^2)$.
According to the classification of \cite[Tables 8.72 and 8.73]{Ho}, $M$ is a subgroup belonging to $\Cl_2 \cup
\Cl_5\cup 
\Cl_6\cup \mathcal{S}$. 
Classes $\Cl_2$, $\Cl_5$ and $\mathcal{S}$ are excluded using Lemmas  \ref{mon11}, \ref{sub11} and 
\ref{S11}.
Since $(xy)^6$ is not a scalar matrix (the element of position $(5,1)$ is $1$), we can apply \cite[Lemma 2.3]{SU35} to prove that $H$ is 
not contained in a maximal subgroup of class $\Cl_6$.
We conclude that $H=\SU_{11}(q^2)$.

Now, suppose $q>2$.  Clearly, there are  $\varphi(q^2-1)$ different generators $a$ of $\F_{q^2}^*$; on the other hand,
there are at most $7q$ values $a$ such that $f_1(a)f_2(a)=0$.
For $q\in \{16,25,27 \}$ and for $q\geq 31$ we get $\varphi(q^2-1)>  7q$ and hence we can choose $a\in \F_{q^2}^*$
such that
\eqref{11g} is satisfied.
For each of the remaining values of $q$, 
in Table \ref{tab11} we exhibit a value $a$ which satisfies \eqref{11g}.
The inequality $\varphi(q^2-1)>  7q$ can be proved by direct computations for $q\in \{16,25,27 \}$ and for $31\leq q \leq 337$; for 
$q>337$ we use the fact that $\varphi(q^2-1) > (q^2-1)^{2/3}$, see \cite[Lemma 2.1]{SU35}.
\end{proof}

\begin{table}[ht]
$$\begin{array}{|cc|cc|cc|}\hline
q & m(t) & q & m(t) & q & m(t)\\\hline
3,5 & t^2 - t +11  & 4,9 &  t^4+t^3+t^2+t+1 & 7,13,17,19,29 & t^2+t+3 \\
8 &  t^6+t^3+1 & 11 & t^2+5t+3 &  23 & t^2+3t+3 \\\hline
  \end{array}$$
\caption{Minimum polynomial $m(t)$ of $a$ over $\F_p$.}\label{tab11}
\end{table}

\section{Conclusions}

\begin{proof}[Proof of Theorem \ref{main}]
For $3\leq n\leq 7$, the result was already proved in \cite{SU7,SU6,SU35,SU4}.
So, suppose $n\geq 9$ and $n\neq 11$.
If  $q= 2,3,5,7,8$ the statement follows from Proposition \ref{23578} and for the other values of $q$ we apply Theorem \ref{p<>3}.
The statement for $n=8$ follows from Propositions \ref{23578} and  \ref{p8no3}.
Finally, the statement for  $n=11$ follows from Propositions \ref{23578} and \ref{pr11}.
\end{proof}

\end{document}